\documentclass[a4paper,12pt]{article}
\usepackage{amsmath,amssymb,amsthm,amsxtra}
\newcommand{\C}{{\mathbb C}}
\newcommand{\N}{{\mathbb N}}
\newcommand{\R}{{\mathbb R}}

\newcommand{\abs}[2][\empty]{\ifx#1\empty\left|#2\right|%
\else#1\vert #2 #1\vert\fi}
\newcommand{\caninf}{\rho}
\newcommand{\Cnt}[1][]{{\cal C}^{#1}}
\newcommand{\comp}{\circ}
\newcommand{\conv}{\star}
\newcommand{\csub}{\subset\subset}
\newcommand{\defstyle}[1]{{\em #1}}
\newcommand{\eps}{\varepsilon}
\renewcommand{\Im}{\mathop{\mathrm{Im}}}
\renewcommand{\implies}{\Rightarrow}
\newcommand{\interior}[1]{{#1}^\circ}
\newcommand{\interl}{\mathop{\mathrm{interl}}}
\newcommand{\inv}[1]{{#1}^{-1}}

\newcommand{\meas}{\mu}
\newcommand{\norm}[2][\empty]{\ifx#1\empty\left\Vert#2\right\Vert%
\else#1\Vert #2 #1\Vert\fi}
\renewcommand{\Re}{\mathop{\mathrm{Re}}}
\newcommand{\restr}[2]{{{#1}_{|{#2}}}}
\newcommand{\sharpnorm}[2][\empty]{\abs[#1]{#2}_{\mathrm e}}
\newcommand{\supp}{\mathop{\mathrm{supp}}}

\newcommand{\val}{\mathrm{v}}

\newcommand{\Gen}{{\mathcal G}}
\newcommand{\GenC}{\widetilde\C}

\newcommand{\GenR}{\widetilde\R}
\newcommand{\EMod}{{{\mathcal E}_M}}
\newcommand{\Mod}{{\mathcal M}}
\newcommand{\Null}{{\mathcal N}}
\newcommand{\tGen}{\widetilde\Gen}

\newtheorem{thm}{Theorem}[section]
\newtheorem{lemma}[thm]{Lemma}
\newtheorem{prop}[thm]{Proposition}
\newtheorem{df}[thm]{Definition}
\newtheorem{cor}[thm]{Corollary}
\newtheorem{ex}[thm]{Example}

\theoremstyle{definition}
\newtheorem*{rem}{Remark}

\begin{document}
\title{Generalized Analytic Functions on Generalized Domains}
\author{Hans Vernaeve\footnote{Supported by FWF (Austria), grants M949-N18 and  Y237-N13.}\\Dept.\ of Pure Mathematics and Computer Algebra\\Ghent University\\
{\tt hvernaev@cage.ugent.be}}
\date{}
\maketitle

\begin{abstract}
We define the algebra $\tGen(A)$ of Colombeau generalized functions on a subset $A$ of the space of generalized points $\GenR^d$. If $A$ is an open subset of $\GenR^d$, such generalized functions can be identified with pointwise maps from $A$ into the ring of generalized numbers $\GenC$. We study analyticity in $\tGen(A)$, where $A$ is an open subset of $\GenC$. In particular, if the domain is an open ball for the sharp norm on $\GenC$, we characterize analyticity and give a unicity theorem involving the values at generalized points.
\end{abstract}



\section{Introduction}
From the very beginning of the theory of nonlinear generalized functions, holomorphic generalized functions have been studied \cite{Aragona85,Colombeau84,Colombeau85}. More recently, analyticity of pointwise maps $A\subseteq \GenC\to\GenC$ ($A$ open) has been considered \cite{AFJ05,PSV06} in relation with holomorphic generalized functions on an open domain $\Omega\subseteq\C$ (which can be considered as pointwise maps on the set $\widetilde\Omega_c$ of so-called compactly supported generalized points in $\Omega$ \cite[\S1.2.4]{GKOS}). Recently, a theory of integration of generalized functions on generalized subsets (called membranes) has been developed and a generalized Cauchy formula has been proved \cite{AFJO08}. Very soon in the development of the theory, also some striking differences with the classical theory have been noticed. For instance, neither the values of the derivatives of any order of a generalized holomorphic function $f$ at one point, nor an accumulation point of values of $f$ determine $f$ uniquely \cite[\S 8.7]{Colombeau84}. Nevertheless, strong unicity theorems for holomorphic generalized functions have been obtained in \cite{KS06}.

We define the algebra $\tGen(A)$ of generalized functions on a subset $A\subseteq\GenR^d$ in such a way that the traditional Colombeau algebra $\Gen(\Omega)$ on an open subset $\Omega\subseteq\R^d$ coincides with $\tGen(\widetilde \Omega_c)$ (Corollary \ref{cor_Gen_Omega_is_tGen_Omega_c}), and the pointwise actions as a map $\widetilde\Omega_c\to\GenC$ are identical (proposition \ref{prop_GenA_pointvalues}). After establishing some properties of $\tGen(A)$ that extend known results about the traditional Colombeau algebras, such as an analogue of the sheaf property (propostion \ref{prop_sheaf}) and a pointwise invertibility criterium (proposition \ref{prop_invertible} and its corollary), we focus our attention to analyticity in $\tGen(A)$, where $A\subseteq\GenC$ is open. We use a result on complex integration over generalized paths similar to \cite{AFJO08} (proposition \ref{prop_Cauchy_formula}), though our definition of generalized path is slightly different (definition \ref{df_path}). For generalized holomorphic functions on a domain of the form $\{\tilde z\in\GenC: \sharpnorm{\tilde z-\tilde z_0}< r\}$, with $z_0\in\GenC$, $r\in\R^+$ (an open ball for the sharp norm $\sharpnorm{.}$ on $\GenC$), the unicity theorem that is lacking for traditional generalized functions on open domains in $\C$ holds (proposition \ref{thm_acc_point_of_zeroes}). The phenomenon for traditional holomorphic generalized functions on an open domain $\Omega$ in $\C$ can thus be interpreted as a result of the fact that the largest part of $\widetilde\Omega_c$ lies on the `boundary of the convergence disc'. We further characterize analyticity in an open ball for the sharp norm in four different ways (theorem \ref{thm_holomorphic_charac}). Generalized domains are also a natural setting to obtain an analogue of Liouville's theorem (proposition \ref{prop_Liouville}). Apart from developing a tool for modeling singular nonlinear phenomena, our motivation for considering (in particular holomorphic) generalized functions on generalized domains is also to obtain a spectral radius formula in the theory of Banach $\GenC$-algebras \cite{HV_Banach}.

\section{Preliminaries}
Let $E$ be a locally convex vector space over $\C$ with its topology generated by a family of seminorms $(p_i)_{i\in I}$. Then the Colombeau space $\Gen_E:= \Mod_E/\Null_E$ \cite{Garetto2005}, where
\begin{align*}
\Mod_E &= \{(u_\eps)_\eps\in\ E^{(0,1)}: (\forall i\in I) (\exists N\in\N) (p_i(u_\eps)\le \eps^{-N}, \text{ for small }\eps)\}\\
\Null_E &= \{(u_\eps)_\eps\in\ E^{(0,1)}: (\forall i\in I) (\forall m\in\N) (p_i(u_\eps)\le \eps^{m}, \text{ for small }\eps)\}.
\end{align*}
Elements of $\Mod_{E}$ are called moderate, elements of $\Null_{E}$ negligible.
The element of $\Gen_E$ with $(u_\eps)_\eps$ as a representative is denoted by $[(u_\eps)_\eps]$. If $\Omega\subseteq\R^d$ is open and $\Cnt[\infty](\Omega)$ is provided with its usual locally convex topology, i.e., generated by the seminorms $p_{m,K}(u) := \sup_{\abs\alpha\le m, x\in K} \abs{\partial^\alpha u(x)}$ ($m\in\N$, $K\csub\Omega$), then $\Gen(\Omega):=\Gen_{\Cnt[\infty](\Omega)}$ is the so-called (special) algebra of Colombeau generalized functions (cf.\ \cite[\S 1.2]{GKOS}). $\GenR:=\Gen_{\R}$ and $\GenC:=\Gen_\C$ are the so-called Colombeau generalized numbers. We will denote $\caninf:= [(\eps)_\eps]\in\GenR$. \\
For $(x_\eps)_\eps\in(\R^d)^{(0,1)}$, the valuation $\val(x_\eps):= \sup\{b\in\R: \abs{x_\eps}\le \eps^b$, for small $\eps \}$ and the so-called sharp norm $\sharpnorm{x_\eps}:= e^{- \val(x_\eps)}$. For $\tilde x=[(x_\eps)_\eps]\in\GenR^d$, $\val(\tilde x):= \val(x_\eps)\in (-\infty, \infty]$ and $\sharpnorm{\tilde x}:= \sharpnorm{x_\eps}\in [0,+\infty)$ are defined independent of the representative of $\tilde x$. Thus $\GenR^d$ becomes a metric space for the ultrametric $d(\tilde x_1, \tilde x_2):= \sharpnorm{\tilde x_1 - \tilde x_2}$. The corresponding topology is called the sharp topology on $\GenR^d$ \cite{B90,Garetto2005,S92}. Similarly, the sharp topology on $\GenC$ is defined. For $\tilde x=[(x_\eps)_\eps]\in\GenR^d$, we will denote $\abs{\tilde x}:= [(\abs{x_\eps})_\eps]\in\GenR$ (and similarly for $\tilde z\in\GenC$).\\
Let $A_\eps\subseteq \GenR^d$, $\forall \eps\in (0,1)$. Then the set
\[
[(A_\eps)_\eps] := \{\tilde x\in \GenR^d: (\exists \text{ repres.\ $(x_\eps)_\eps$ of } \tilde x) (x_\eps\in A_\eps, \text{ for small }\eps)\}
\]
is called the \defstyle{internal subset} of $\GenR^d$ with representative $(A_\eps)_\eps$ \cite{OV_internal} (and similarly for subsets of $\GenC$). For $A\subseteq \R^d$, we denote $\widetilde A:=[(A)_\eps]$.\\
A subset $A$ of $\GenR^d$ is called sharply bounded if $\sup_{\tilde x\in A}\sharpnorm{\tilde x}<+\infty$. An internal set $A$ is sharply bounded iff $A$ has a sharply bounded representative, i.e., a representative $[(A_\eps)_\eps]$ for which there exists $M\in\N$ such that $\sup_{x\in A_\eps}\abs{x}\le \eps^{-M}$, for small $\eps$ \cite[Lemma 2.4]{OV_internal}.\\
If $\Omega\subseteq\R^d$ is open, then $\widetilde \Omega_c:= \bigcup_{K\csub \Omega}\widetilde K\subseteq \GenR^d$ is the set of compactly supported points in $\Omega$. For $u = [(u_\eps)_\eps]\in\Gen(\Omega)$ and $\tilde x=[(x_\eps)_\eps]\in\widetilde\Omega_c$, the generalized point value $u(\tilde x):=[(u_\eps(x_\eps))_\eps]$ is well-defined (independent of representatives of $u$ and $\tilde x$) \cite[\S 1.2]{GKOS}.

We refer to \cite{GKOS} for further properties of Colombeau generalized functions.

\section{The Colombeau algebra on a subset of $\GenR^d$}
\begin{df}
Let $\emptyset\ne A\subseteq \GenR^d$. We define $\tGen(A)=\EMod(A)/\Null(A)$, where
\begin{align*}
\EMod(A)=\,&\big\{(u_\eps)_\eps\in\Cnt[\infty](\R^d)^{(0,1)}: (\forall\alpha\in\N^d) \big(\forall [(x_\eps)_\eps]\in A\big)
\big((\partial^\alpha u_\eps(x_\eps))_\eps\in\Mod_{\C}\big)\big\},\\
\Null(A)=\,&\big\{(u_\eps)_\eps\in\Cnt[\infty](\R^d)^{(0,1)}: (\forall\alpha\in\N^d) \big(\forall [(x_\eps)_\eps]\in A\big)
\big((\partial^\alpha u_\eps(x_\eps))_\eps\in\Null_{\C}\big)\big\}.
\end{align*}
(Here $\forall [(x_\eps)_\eps]\in A$ means: for each representative $(x_\eps)_\eps$ of an element of $A$.)
\end{df}
Since $\EMod(A)$ is a differential algebra (for the $\eps$-wise operations) and $\Null(A)$ is a differential ideal of $\EMod(A)$, $\tGen(A)$ is a differential algebra.
\begin{df}
Let $\emptyset \ne A\subseteq B\subseteq\GenR^d$. Then the identity map on representatives gives rise to a well-defined map $\restr{.}{A}$: $\tGen(B)\to\tGen(A)$ which we call the \defstyle{restriction} map.
\end{df}

\begin{lemma}\label{lemma_moderate_uniform}
Let $(A_n)_{n\in\N}$ be a decreasing sequence of non-empty, internal, sharply bounded subsets of $\GenR^d$. Let $(u_\eps)_{\eps\in(0,1)}$ be a net of maps $\R^d\to\C$. Then for any sharply bounded representatives $(A_{n,\eps})_\eps$ of $A_n$,
\begin{multline*}
(u_\eps(x_\eps))_\eps\in\Mod_\C, \forall [(x_\eps)_\eps]\in \bigcap_{n\in\N} A_n\iff (\exists N\in\N) \big(\sup_{x\in A_{N,\eps} + \eps^N}\abs{u_\eps(x)}\le\eps^{-N},\text{ for small }\eps\big).
\end{multline*}
\end{lemma}
\begin{proof}
$\Rightarrow$: By \cite[Prop.\ 2.9]{OV_internal}, for each $m\in\N$, there exists $\eta_m\in (0,1)$ such that for each $\eps\le\eta_m$ and $x\in A_{m,\eps}$, $d(x, A_{k,\eps})\le\eps^m$, for each $k\le m$. W.l.o.g., $(\eta_m)_{m\in\N}$ decreasingly tends to $0$. By contraposition, let
\[(\forall n\in\N)
(\forall\eta\in(0,1)) (\exists\eps\le\eta)
\big(\sup_{x\in A_{n,\eps} + \eps^n}\abs{u_\eps(x)}>\eps^{-n}\big).\]
Then we can find a strictly decreasing sequence $(\eps_n)_{n\in\N}$ and $x_{\eps_n}\in A_{n,\eps_n} + \eps_n^n$ such that $\eps_n\le\eta_n$ and $\abs{u_{\eps_n}(x_{\eps_n})}>\eps_n^{-n}$, $\forall n\in\N$. 
Choose $x_\eps \in A_{m,\eps}$, if $\eta_{m+1}<\eps\le \eta_m$ and $\eps\notin\{\eps_n: n\in\N\}$. Then for each $n\in\N$, $(d(x_\eps, A_{n,\eps}))_\eps\in\Null_\R$. By \cite{OV_internal}, $\tilde x:=[(x_\eps)_\eps]\in \bigcap_{n\in\N} A_n$ ($(x_\eps)_\eps$ is moderate, since $(A_{n,\eps})_\eps$ are sharply bounded). Yet $(u_\eps(x_\eps))_\eps\notin\Mod_\C$.\\
$\Leftarrow$: let $[(x_\eps)_\eps]\in \bigcap_{n\in\N} A_n$. Let $N\in\N$ as in the statement. By \cite{OV_internal}, $(d(x_\eps, A_{n,\eps}))_\eps\in\Null_\R$ for each $n\in\N$. In particular, $x_\eps\in A_{N,\eps} + \eps^N$ for small $\eps$. Hence $\abs{u_\eps(x_\eps)}\le \sup_{x\in A_{N,\eps} + \eps^N}\abs{u_\eps(x)} \le\eps^{-N}$ for small $\eps$.
\end{proof}

\begin{prop}\label{prop_EModA_monad}
Let $(A_n)_{n\in\N}$ be a decreasing sequence of non-empty, internal, sharply bounded subsets of $\GenR^d$. Then for any sharply bounded representatives $(A_{n,\eps})_\eps$ of $A_n$,
\begin{multline*}
\EMod\big(\smash{\bigcap_{n\in\N}} A_n\big) = \big\{(u_\eps)_\eps\in\Cnt[\infty](\R^d)^{(0,1)}: (\forall\alpha\in\N^d) (\exists N\in\N)\\
\big(\sup_{x\in A_{N,\eps} + \eps^N}\abs{\partial^\alpha u_\eps(x)}\le\eps^{-N},\text{ for small }\eps\big)\big\}.
\end{multline*}
\end{prop}
\begin{proof}
By lemma \ref{lemma_moderate_uniform}.
\end{proof}
\begin{cor}\label{cor_EModA_internal}
Let $\emptyset\ne A\subseteq\GenR^d$ be internal and sharply bounded. Then for each sharply bounded representative $(A_\eps)_\eps$ of $A$,
\begin{multline*}
\EMod(A)=\big\{(u_\eps)_\eps\in\Cnt[\infty](\R^d)^{(0,1)}: (\forall\alpha\in\N^d) (\exists N\in\N)\\
\big(\sup_{x\in A_\eps + \eps^N}\abs{\partial^\alpha u_\eps(x)}\le\eps^{-N},\text{ for small }\eps\big)\big\}.
\end{multline*}
\end{cor}

\begin{prop}\label{prop_GenA_pointvalues}
Let $\emptyset\ne A\subseteq\GenR^d$. Let $u=[(u_\eps)_\eps]\in\tGen(A)$.
\begin{enumerate}
\item For $\tilde x=[(x_\eps)_\eps]\in A$, $u(\tilde x) = [(u_\eps(x_\eps))_\eps]\in \GenC$ is well-defined (independent of representatives).
\item If $\tilde x\in \interior{A}$ and $u(\tilde y)=0$, for each $\tilde y$ in a sharp neighbourhood of $\tilde x$, then $\partial^\alpha u(\tilde x)=0$, $\forall\alpha\in\N^d$.
\item If $A$ is open, then $u=0$ iff $u(\tilde x) = 0$, $\forall \tilde x\in A$.
\end{enumerate}
\end{prop}
\begin{proof}
(1) To prove independence of the representative of $\tilde x$, let $\tilde x =[(x_\eps)_\eps]=[(y_\eps)_\eps]$. By corollary \ref{cor_EModA_internal}, since $\{\tilde x\}$ is internal and sharply bounded, there exists $N\in\N$ such that $\sup_{\abs{x-x_\eps}\le \eps^N}\abs{\nabla u_\eps(x)} \le\eps^{-N}$, for small $\eps$. Hence there exist $y'_\eps\in\R^d$ with $\abs{y'_\eps - x_\eps}\le\abs{y_\eps-x_\eps}$ such that \[\abs{u_\eps(y_\eps)-u_\eps(x_\eps)}\le \abs{y_\eps - x_\eps} \abs{\nabla u_\eps(y'_\eps)}\le\eps^{-N}\abs{y_\eps - x_\eps},\]
for small $\eps$.\\
(2) Let $N\in\N$ such that $\tilde y\in A$ and $u(\tilde y)=0$, for each $\tilde y\in\GenR$ with $\abs{\tilde y - \tilde x}\le \eps^N$. Let $\tilde x=[(x_\eps)_\eps]$. By contraposition, $(\sup_{\abs{x-x_\eps}\le\eps^N}\abs{u_\eps(x)})_\eps\in\Null_\C$. Since $(u_\eps)_\eps\in\EMod(A)$, we find as in part~1 for each $\alpha\in\N^d$ some $N\in\N$ such that $\sup_{\abs{x-x_\eps}\le \eps^N}\abs{\partial^\alpha u_\eps(x)} \le\eps^{-N}$, for small $\eps$. The statement follows analogously to \cite[Thm.~1.2.3]{GKOS}.\\
(3) $\subseteq$: clear.\\
$\supseteq$: let $\alpha\in\N^d$ and $\tilde x\in A$. By part~2, $\partial^\alpha u(\tilde x)=0$. By part~1, $(\partial^\alpha u_\eps(x_\eps))_\eps\in\Null_\C$ for any representative $[(x_\eps)_\eps]$ of $\tilde x$.
\end{proof}

\begin{prop}\label{prop_NullA_internal}
Let $\emptyset\ne A\subseteq\GenR^d$ be internal and sharply bounded. Then for each sharply bounded representative $(A_\eps)_\eps$ of $A$,
\begin{align*}
\Null(A)=&\,\big\{(u_\eps)_\eps\in\Cnt[\infty](\R^d)^{(0,1)}: (\forall\alpha\in\N^d) (\forall m\in\N)\\
&\,(\exists N\in\N)
\big(\sup_{x\in A_\eps+\eps^N}\abs{\partial^\alpha u_\eps(x)}\le\eps^m, \text{ for small }\eps\big)\big\}\\
=&\,\big\{(u_\eps)_\eps\in\EMod(A): (\forall\alpha\in\N^d)
\big(\big(\sup_{x\in A_\eps}\abs{\partial^\alpha u_\eps(x)}\big)_\eps \in\Null_\R \big)\big\}.
\end{align*}
\end{prop}
\begin{proof}
(1)$\subseteq$(2): by contraposition (as in proposition \ref{prop_EModA_monad}).\\
(2)$\subseteq$(3): clear by corollary \ref{cor_EModA_internal}.\\
(3)$\subseteq$(1): let $\alpha\in\N^d$ and $\tilde x\in A$. Then $\tilde x=[(a_\eps)_\eps]$, with $a_\eps\in A_\eps$ for small $\eps$. By hypothesis, $(\partial^\alpha u_\eps(a_\eps))_\eps\in\Null_\C$. By proposition \ref{prop_GenA_pointvalues}(1), $\partial^\alpha u(\tilde x)=0$ and $(\partial^\alpha u_\eps(x_\eps))_\eps\in\Null_\C$ for any representative $[(x_\eps)_\eps]$ of $\tilde x$.
\end{proof}

\begin{lemma}
Let $\emptyset\ne B_\lambda\subseteq\GenR^d$, for each $\lambda\in\Lambda$ (where $\Lambda$ is some index set). Then $\EMod(\bigcup_{\lambda\in\Lambda} B_\lambda) = \bigcap_{\lambda\in\Lambda}\EMod(B_\lambda)$ and $\Null(\bigcup_{\lambda\in\Lambda} B_\lambda) = \bigcap_{\lambda\in\Lambda}\Null(B_\lambda)$.
\end{lemma}
\begin{proof}
By definition.
\end{proof}

\begin{cor}\label{cor_GenA_union_of_internal}
Let $A=\bigcup_{\lambda\in\Lambda} B_\lambda\subseteq\GenR^d$, where each $B_\lambda$ is nonempty, internal and sharply bounded. Let $(B_{\lambda,\eps})_\eps$ be a sharply bounded representative of $B_\lambda$, for each $\lambda$. Then
\begin{align*}
\EMod(A)=\,&\big\{(u_\eps)_\eps\in\Cnt[\infty](\R^d)^{(0,1)}: (\forall\alpha\in\N^d) (\forall \lambda\in\Lambda)\\
&(\exists N\in\N)
\big(\sup_{x\in B_{\lambda,\eps} + \eps^N}\abs{\partial^\alpha u_\eps(x)}\le\eps^{-N},\text{ for small }\eps\big)\big\}.\\
\Null(A)=\,&\big\{(u_\eps)_\eps\in\EMod(A): (\forall\alpha\in\N^d) (\forall \lambda\in\Lambda)
\big(\big(\sup_{x\in B_{\lambda,\eps}}\abs{\partial^\alpha u_\eps(x)}\big)_\eps\in\Null_\R\big)\big\}.
\end{align*}
\end{cor}
\begin{proof}
Combine the previous lemma with corollary \ref{cor_EModA_internal} and proposition \ref{prop_NullA_internal}.
\end{proof}
\begin{cor}\label{cor_Gen_Omega_is_tGen_Omega_c}
Let $\Omega$ be an open subset of $\R^d$. Then $\Gen(\Omega)=\tGen(\widetilde\Omega_c)$.
\end{cor}
\begin{proof}
Since $\widetilde\Omega_c=\bigcup_{\emptyset\ne K\csub\Omega}\widetilde K$, and since elements of $\Gen(\Omega)$ have representatives in $\Cnt[\infty](\R^d)^{(0,1)}$ (by a cut-off procedure), this follows by corollary \ref{cor_GenA_union_of_internal}.
\end{proof}
Similarly, since $\GenR^d=\bigcup_{n\in\N}\{x\in\GenR^d: \abs{x}\le\caninf^{-n}\}$, $\tGen(\GenR^d)$ coincides with the definition of $\Gen(\GenR^d)$ given in \cite{HV_pointchar}.

Another approach to generalized functions on subsets of $\GenR^d$ could use nets of functions defined on subsets of $\R^d$ only. The following lemma relates such an approach to our definitions.
\begin{lemma}\label{lemma_tGen_by_representatives}
For each $\eps\in (0,1)$, let $\Omega_\eps\subseteq\R^d$ be open. Let $A_{m,\eps} = \{x\in\R^d: d(x,\R^d\setminus\Omega_\eps)\ge \eps^m\}$, for each $m$ and $\eps\in(0,1)$. Let $u_\eps\in\Cnt[\infty](\Omega_\eps)$ such that for each $\alpha\in\N^d$ and $m\in\N$, there exists $N\in\N$ such that $\sup_{x\in A_{m,\eps}, \abs{x}\le\eps^{-m}}\abs{\partial^\alpha u_\eps(x)}\le \eps^{-N}$, for small $\eps$.
Let $A=\bigcup_{m\in\N} [(A_{m,\eps})_\eps]$. Then there exists a unique $u\in\tGen(A)$ such that $\partial^\alpha u(\tilde x)=[(\partial^\alpha u_\eps(x_\eps))_\eps]$, for each $\tilde x\in A$, for each representative $(x_\eps)_\eps$ of $\tilde x$ and $\alpha\in\N^d$.
\end{lemma}
\begin{proof}
For $m\in\N$ and $\eps\in (\frac{1}{m+1},\frac{1}{m}]$, let $\chi_\eps\in\Cnt[\infty](\R^d)$ with $\chi_\eps(x)=1$, if $x\in A_{m,\eps}$ and $\chi_\eps(x)=0$, if $x\notin A_{m+1,\eps}$. Then $v_\eps:=\chi_\eps u_\eps\in\Cnt[\infty](\R^d)$, $\forall\eps$. Let $\tilde x\in A$. Then there exists $m\in\N$ and a representative $(x_\eps)_\eps$ with $x_\eps\in A_{m,\eps}$, for each $\eps$. Hence for any representative $(x'_\eps)_\eps$, $x'_\eps\in A_{m+1,\eps}$, for small $\eps$, and $\partial^\alpha v_\eps(x_\eps') = \partial^\alpha u_\eps(x'_\eps)$, for small $\eps$ and for $\alpha\in\N^d$. Thus $u:=[(v_\eps)_\eps]\in\tGen(A)$ and $\partial^\alpha u(\tilde x) = [(\partial^\alpha u_\eps(x'_\eps))_\eps]$. Unicity of $u$ follows directly from the definition of $\Null(A)$.
\end{proof}
Under the conditions of the previous lemma, we will (loosely) say that $[(u_\eps)_\eps]\in\tGen(A)$.

\begin{lemma}\label{lemma_internal_nbd}
Let $A\subseteq \GenR^d$ be internal and sharply bounded and let $B\subseteq\GenR^d$ be an internal sharp neighbourhood of $A$. Then:
\begin{enumerate}
\item There exists $m\in\N$ such that for each $\tilde a\in A$, $\overline B(\tilde a,\caninf^m)=\{\tilde x\in\GenR^d: \abs{\tilde x- \tilde a}\le \caninf^m\}\subseteq B$.
\item Given a sharply bounded representative $(A_\eps)_\eps$ of $A$, we can find a representative $(B_\eps)_\eps$ of $B$ such that $A_\eps + \eps^m \subseteq B_\eps$, $\forall\eps$.
\end{enumerate}
\end{lemma}
\begin{proof}
(1) Let $A=[(A_\eps)_\eps]$ and $B=[(B_\eps)_\eps]$. W.l.o.g., $A\ne\emptyset$ and $(A_\eps)_\eps$ is a sharply bounded representative. Suppose that for each $n\in\N$, there exists $\tilde a_n = [(a_{n,\eps})_\eps]\in A$ and $\tilde x_n= [(x_{n,\eps})_\eps]\in \GenR^d\setminus B$ with $\abs{\tilde x_n - \tilde a_n}\le \caninf^n$. W.l.o.g., $a_{n,\eps}\in A_\eps$, $\forall\eps$. By \cite[Prop.\ 2.1]{OV_internal}, $(d(x_{n,\eps}, B_\eps))_\eps\notin\Null_\R$, so for each $n\in\N$, there exists $k_n\in\N$ such that for each $\eta\in(0,1)$, there exists $\eps\le\eta$ with $d(x_{n,\eps}, B_\eps)\ge \eps^{k_n}$. We can thus find $\eps_{n,m}\in (0,1/m)$, for each $n,m\in\N$ (by enumerating $(\eps_{n,m})_{n,m\in\N}$, we can ensure that all $\eps_{n,m}$ are different) and $x_{n,\eps_{n,m}}\in\R^d$ with $d(x_{n,\eps_{n,m}}, B_{\eps_{n,m}})\ge \eps_{n,m}^{k_n}$ and $\abs[]{x_{n,\eps_{n,m}} - a_{n,\eps_{n,m}}}\le 2 \eps_{n,m}^n$, for each $n,m\in\N$. Let $a_{\eps_{n,m}}:= a_{n,\eps_{n,m}}$, for each $n,m\in\N$ and $a_\eps\in A_\eps$ arbitrary, if $\eps\notin\{\eps_{n,m}: n,m\in\N\}$. Let $x_{n,\eps}:= a_\eps$, if $\eps\notin\{\eps_{n,m}: m\in\N\}$, for each $n\in\N$. As $(A_\eps)_\eps$ is sharply bounded, $\tilde a:= [(a_\eps)_\eps]\in A$ and $\abs{\tilde x_n - \tilde a}\le 2 \caninf^n$, for each $n\in\N$. Since $B$ is a neighbourhood of $A$, $\tilde x_n\in B$ for large $n$, a contradiction.\\
(2) As $(A_\eps)_\eps$ is sharply bounded, it follows that $[(A_{\eps} + \eps^{m})_\eps]\subseteq B$. Since $[((A_{\eps} + \eps^{m})\cup B_{\eps})_\eps]$ is the smallest internal set containing $[(A_{\eps} + \eps^m)_\eps]$ and $B$ \cite[Prop. 2.8]{OV_internal}, $B = [((A_{\eps} + \eps^m)\cup B_{\eps})_\eps]$.
\end{proof}

Recall that the interleaved closure of $A\subseteq\GenR^d$ \cite[Lemma 2.7.]{OV_internal} is the set
\[
\interl(A):=\big\{\sum_{j=1}^m e_{S_j} x_j: m\in\N, \{S_1,\dots,S_m\} \text{ partition of }(0,1), x_j\in A\big\}.
\]
\begin{prop}[Generalized sheaf property]\label{prop_sheaf}\leavevmode
\begin{enumerate}
\item Let $\Omega\subseteq\GenR^d$ be a union of an increasing sequence $(A_{n})_{n\in\N}$ of internal sets with $A_{n+1}$ a neighbourhood of $A_{n}$, for each $n$ (hence, in particular, $\Omega$ open). If $u_n\in\tGen(A_n)$ and $\restr{u_{n+1}}{A_n}= u_n$, for each $n\in\N$, then there exists a unique $u\in\tGen(\Omega)$ such that $\restr{u}{A_n}=u_{n}$, for each $n\in\N$.
\item For each $m\in\N$, let $\Omega_m\subseteq\GenR^d$ be a union of an increasing sequence $(A_{m,n})_{n\in\N}$ of internal sets with $A_{m,n+1}$ a neighbourhood of $A_{m,n}$, for each $n$. Let $u_m\in\tGen(\Omega_m)$, for each $m\in\N$ such that $\restr{u_m}{\Omega_m\cap\Omega_{m'}} = \restr{u_{m'}}{\Omega_m\cap\Omega_{m'}}$, for each $m$, $m'\in\N$. Let $\Omega = \interl(\bigcup_{m\in\N}\Omega_m)$. Then there exists a unique $u\in\tGen(\Omega)$ such that $\restr{u}{\Omega_m} = u_m$, for each $m\in\N$.
\end{enumerate}
\end{prop}
\begin{proof}
Let $\Omega_m$, $\Omega$ and $A_{m,n}$ as in (2). Let $u_{m,n}=[(u_{m,n,\eps})_\eps]\in\tGen(A_{m,n})$, for each $m,n\in\N$ such that $\restr{u_{m,n}}{A_{m,n}\cap A_{m',n'}}=\restr{u_{m',n'}}{A_{m,n}\cap A_{m',n'}}$, for each $m$, $m'$, $n$, $n'\in\N$. It is sufficient to show that there exists a unique $u\in\tGen(\Omega)$ such that $\restr{u}{A_{m,n}}=u_{m,n}$, for each $m,n,\in\N$. Let $A_{m,n}=[(A_{m,n,\eps})_\eps]$. Since $\Omega_m = \bigcup (A_{m,n}\cap B(0,\caninf^n))$, we may assume that all $A_{m,n}$ (and hence $(A_{m,n,\eps})_\eps$ \cite[Lemma 2.4]{OV_internal}) are sharply bounded.
We may also assume that all $A_{m,n,\eps}$ are closed \cite[Cor. 2.2]{OV_internal}. Let $m,n\in\N$. By lemma \ref{lemma_internal_nbd}, we may assume that there exist $k_{m,n}\in\N$ such that $A_{m,n,\eps}+\eps^{k_{m,n}}\subseteq A_{m,n+1,\eps}$, for each $m,n,\eps$.\\ 
Let $B_n=[(B_{n,\eps})]$ with $B_{n,\eps} = A_{1,n,\eps}\cup \cdots\cup A_{n,n,\eps}$, for each $n\in\N$ and $\eps\in (0,1)$.
Let $\theta\in \Cnt[\infty](\R^d)$ with $\theta(x)=0$, if $\abs x\ge 1$ and $\theta(x)\ge 0$, for each $x\in\R^d$ with $\int_{\R^d}\theta = 1$ and let $\theta_r(x):=\inv r \theta(\inv r x)$, for $r\in\R^+$. Let $\chi_A$ denote the characteristic function of a set $A\subseteq\R^d$. For each $m,n,\eps$, let $\phi_{m,n,\eps}= \chi_{A_{m,n+2,\eps}\setminus B_{n-1,\eps}}\conv \theta_{\eps^{l_{m,n}}}$, where $l_{m,n}=\max_{i\le m, j\le n+2}k_{i,j}$. Then $\phi_{m,n,\eps}(x)=1$, for each $x\in A_{m,n+1,\eps}\setminus B_{n,\eps}$ and $\supp\phi_{m,n,\eps}\subseteq A_{m,n+3,\eps}\setminus B_{n-2,\eps}$. Further, $\sup_{x\in\R^d}\abs{\partial^\alpha \phi_{m,n,\eps}(x)}\le \eps^{-l_{m,n}\abs\alpha}\int_{\R^d}\abs{\partial^\alpha\theta}$ by the properties of the convolution. Let $\phi_\eps:= \sum_{m,n\in\N, m\le n+1}\phi_{m,n,\eps}$. Then $\phi_\eps\in\Cnt[\infty](\bigcup_{n\in\N} B_{n,\eps})$ and for each $n$, $(\sup_{x\in B_{n,\eps}}\abs{\partial^\alpha \phi_\eps(x)})_\eps\in\Mod_\R$, since $\supp\phi_{m',n',\eps}\cap B_{n,\eps}\ne\emptyset$ only for $n'\le n+2$. Also $\phi_\eps(x)\ge 1$, for each $x\in \bigcup_{m,n\in\N, m\le n+1} (A_{m,n+1,\eps}\setminus B_{n,\eps}) = \bigcup_{n\in\N} B_{n,\eps}$ (since $B_{n+1,\eps}\setminus B_{n,\eps}= \bigcup_{m\le n+1} (A_{m,n+1,\eps}\setminus B_{n,\eps})$, for each $n$).
Let $\psi_{m,n,\eps}:= \phi_{m,n,\eps}/\phi_\eps\in \Cnt[\infty](\R^d)$. Then $\sum_{m,n\in\N, m\le n+1} \psi_{m,n,\eps}(x)=1$, for each $x\in \bigcup_{n\in\N} B_{n,\eps}$. Since $\sup_{x\in B_{n,\eps}}\abs{1/\phi_\eps(x)}\le 1$, for each $n$, we find $M_{m,n}\in\N$ such that
\[
\sup_{x\in\R^d} \abs{\partial^\alpha\psi_{m,n,\eps}(x)} = \sup_{x\in B_{n+3,\eps}} \abs{\partial^\alpha\psi_{m,n,\eps}(x)}
\le\eps^{-M_{m,n}},
\]
for small $\eps$. Let $u_\eps := \sum_{m,n\in\N, m\le n+1} \psi_{m,n,\eps} \cdot u_{m,n+3,\eps} \in \Cnt[\infty](\bigcup_{n\in\N} B_{n,\eps})$, for each $\eps$.\\
Let $N\in\N$ and $\alpha\in\N^d$. Then there exists $M\in\N$ such that
\begin{multline*}
\sup_{x\in B_{N,\eps}} \abs{\partial^\alpha u_\eps(x)}\le \sum_{m\le n+1\le N+3}\sup_{x\in B_{N,\eps}} \abs{\partial^\alpha (\psi_{m,n,\eps}\cdot u_{m,n+3,\eps})(x)}\\
\le \sum_{m\le n+1\le N+3}\sup_{x\in A_{m,n+3,\eps}} \abs{\partial^\alpha (\psi_{m,n,\eps}\cdot u_{m,n+3,\eps})(x)} \le\eps^{-M},
\end{multline*}
for small $\eps$, since $\supp\psi_{m,n,\eps} \subseteq A_{m,n+3,\eps}$, and by corollary \ref{cor_EModA_internal}.
As in lemma \ref{lemma_tGen_by_representatives},
we find a unique $u\in\tGen(\bigcup_{n} B_n)$ with $\partial^\alpha u(\tilde x)=[(\partial^\alpha u_\eps(x_\eps))_\eps]$, for each $\tilde x\in \bigcup_n B_n$ and $\alpha\in\N^d$. Let $\tilde x = [(x_\eps)_\eps]\in A_{m,n}$. W.l.o.g., $x_\eps\in A_{m,n,\eps}$, for each $\eps$.
\[\abs{u_\eps(x_\eps) - u_{m,n,\eps}(x_\eps)}\le \sum_{m'\le n' + 1\le \max(m,n) + 3} \abs{\psi_{m',n',\eps}(x_\eps)}\abs{u_{m',n'+3,\eps}(x_\eps) - u_{m,n,\eps}(x_\eps)}\in\Null_\R,\]
by the coherence property, since $\supp\psi_{m',n',\eps}\subseteq A_{m',n'+3,\eps}$. Hence $u(\tilde x)=u_{m,n}(\tilde x)$.
By proposition \ref{prop_GenA_pointvalues}(2), also $\partial^\alpha u(\tilde x)= \partial^\alpha u_{m,n}(\tilde x)$, for each $\alpha\in\N^d$, since $u_{m,n}(\tilde x)= u_{m,n+1}(\tilde x)$ and $\tilde x\in \interior {(A_{m,n+1})}$. Hence $\restr{u}{A_{m,n}}=u_{m,n}$ by the definition of $\Null(A_{m,n})$.\\
Finally, let $\tilde x\in \interl(\bigcup_{m\in\N} \Omega_m)$, i.e., $\tilde x = \sum_{j=1}^M \tilde x_j e_{S_j}$, for some $M\in\N$, a partition $\{S_1,\dots, S_M\}$ of $(0,1)$ and $\tilde x_j\in \Omega_{m_j}$, for some $m_j\in\N$. Then there exists $n\ge \max_{j} m_j$ such that $\tilde x_j\in A_{m_j,n}\subseteq B_n$, for each $j$. Since $B_n$ is internal, $\tilde x\in \interl(B_n)=B_n$. Hence $u\in\tGen(\interl(\bigcup_{m\in\N} \Omega_m))$.
\end{proof}

The following lemma shows that the sheaf property of $\Gen(\Omega)$, $\Omega\subseteq\R^d$ \cite[Thm.\ 1.2.4]{GKOS} can be viewed a special case of proposition \ref{prop_sheaf} (in view of the fact that every open cover of an open $\Omega\subseteq\R^d$ has a countable subcover).
\begin{lemma}
Let $\Omega_\lambda\subseteq \R^d$ be open, for $\lambda\in \Lambda$ and let $\Omega=\bigcup_{\lambda\in\Lambda} \Omega_\lambda$. Then $\widetilde \Omega_c = \interl(\bigcup_{\lambda\in\Lambda} (\Omega_\lambda)\sptilde_c)$.
\end{lemma}
\begin{proof}
$\subseteq$: let $\tilde x\in\widetilde \Omega_c$. There exists $K\csub\Omega$ such that $x\in\widetilde K$. As $K$ is compact, $K\subseteq \bigcup_{\lambda\in F} \Omega_m$, for some finite $F\subseteq \Lambda$. Let $\tilde x=[(x_\eps)_\eps]$ with $x_\eps\in K$, for each $\eps$. Then
\[(\exists N\in\N) (\exists \eps_0\in (0,1)) (\forall \eps\le\eps_0) (\exists \lambda\in F) (d(x_\eps, \R^d\setminus \Omega_\lambda)\ge 1/N),\]
since otherwise, we can construct a decreasing sequence $(\eps_n)_{n\in\N}$ tending to $0$ such that for each $n\in\N$ and $\lambda\in F$, $d(x_{\eps_n}, \R^d\setminus\Omega_\lambda) < 1/n$. As $K$ is compact, a subsequence $x_{\eps_{n_k}}$ would converge to $x_0\in K$. But then $x_0\in \overline{\R^d\setminus\Omega_\lambda} = \R^d\setminus\Omega_\lambda$, for each $\lambda\in F$, contradicting $K\subseteq \bigcup_{\lambda\in F} \Omega_\lambda$. Hence $\tilde x\in \interl(\bigcup_{\lambda\in F} (\Omega_\lambda)\sptilde_c)$.\\
$\supseteq$: let $\tilde x\in \interl(\bigcup_{\lambda\in \Lambda} (\Omega_\lambda)\sptilde_c)$, i.e., $\tilde x = \sum_{j=1}^M \tilde x_j e_{S_j}$, for some $M\in\N$, a partition $\{S_1,\dots, S_M\}$ of $(0,1)$ and $\tilde x_j\in \widetilde K_j$, for some $K_j\csub\Omega$. Then $K:=\bigcup_{j=1}^M K_j\csub \Omega$ and $\tilde x\in \widetilde K$.
\end{proof}

\begin{lemma}\label{lemma_pointwise_zero_internal}
Let $\emptyset\ne A\subseteq\GenR^d$ be internal and sharply bounded, $u=[(u_\eps)_\eps]\in\tGen(A)$. Let $(A_\eps)_\eps$ be a sharply bounded representative of $A$. Then $u(\tilde x)=0$, $\forall \tilde x\in A$ iff $(\sup_{x\in A_\eps}\abs{u_\eps(x)})_\eps\in\Null_\R$.
\end{lemma}
\begin{proof}
$\Rightarrow$: If the conclusion is not true, we find $m\in\N$, a decreasing sequence $(\eps_n)_{n\in\N}$ tending to $0$ and $x_{\eps_n}\in A_{\eps_n}$ such that $\abs{u_{\eps_n}(x_{\eps_n})}\ge \eps_n^m$, for each $n$. Let $x_\eps\in A_\eps$ arbitrary if $\eps\notin\{\eps_n: n\in\N\}$. As $(A_\eps)_\eps$ is sharply bounded, $\tilde x:=[(x_\eps)_\eps]\in A$, and $u(\tilde x)=0$ by assumption, contradicting $\abs{u_{\eps_n}(x_{\eps_n})}\ge \eps_n^m$, for each $n$.\\
$\Leftarrow$: clear.
\end{proof}

\begin{df}(cf.\ \cite{HV_pointchar})
Let $A\subseteq\GenR^d$. Then $\tGen^\infty(A)=\{u\in\tGen(A): (\forall \tilde x\in A) (\exists N\in\N) (\forall\alpha\in\N^d) (\abs{\partial^\alpha u(\tilde x)}\le\caninf^{-N})$.
\end{df}

\begin{prop}\label{prop_GeninftyA_internal}
Let $\emptyset\ne A\subseteq\GenR^d$ be internal and sharply bounded. Let $u=[(u_\eps)_\eps]\in\tGen(A)$. Let $(A_\eps)_\eps$ be a sharply bounded representative of $A$. Then $u\in\tGen^\infty(A)$ iff
\[
(\exists N\in\N) (\forall\alpha\in\N^d) \big(\sup_{x\in A_\eps}\abs{\partial^\alpha u_\eps(x)}\le\eps^{-N}, \text{ for small }\eps\big).
\]
\end{prop}
\begin{proof}
$\Rightarrow$: (cf.\ \cite[Prop.~5.3]{HV_pointchar}).
Supposing the conclusion is not true, we find $\alpha_n\in\N^d$ (for each $n\in\N$), $\eps_{n,m}\in (0,1/m)$ (for each $n,m\in\N$) (by enumerating the countable family $(\eps_{n,m})_{n,m}$, we can successively choose the $\eps_{n,m}$ in such a way that they are all different) and $x_{\eps_{n,m}}\in  A_{\eps_{n,m}}$ with $\abs{\partial^{\alpha_n}u_{\eps_{n,m}}(x_{\eps_{n,m}})} > \eps_{n,m}^{-n}$, $\forall n,m\in\N$. Choose $x_\eps \in A_\eps$ arbitrary, if $\eps\notin\{\eps_{n,m}: n,m\in\N\}$ is sufficiently small ($A_\eps\ne\emptyset$ for small $\eps$ since $A\ne \emptyset$). Then $\tilde x=[(x_\eps)_\eps]\in A$ (moderateness is guaranteed since $(A_\eps)_\eps$ is sharply bounded). By hypothesis, there exists $N\in\N$ such that for each $\alpha\in\N^d$, $\abs{\partial^\alpha u(\tilde x)}\le\caninf^{-N}$. This contradicts the fact that for a fixed $n>N$, $\lim_{m\to\infty}\eps_{n,m}=0$ and $\abs{\partial^{\alpha_n} u_{\eps_{n,m}}(x_{\eps_{n,m}})}> \eps_{n,m}^{-n}$, $\forall m\in\N$.\\
$\Leftarrow$: clear.
\end{proof}

\begin{prop}\label{prop_composition}
Let $A\subseteq \GenR^d$ and $u=[(u_\eps)_\eps]\in\tGen(A)$. Let $u(A)=\{u(\tilde x): \tilde x\in A\}\subseteq B\subseteq \GenC$. Let $v=[(v_\eps)_\eps]\in\tGen(B)$. Then $v\comp u:=[(v_\eps\comp u_\eps)_\eps]\in\tGen(A)$.
\end{prop}
\begin{proof}
By definition, $(v_\eps\comp u_\eps)_\eps\in\Cnt[\infty](\R^d)$. For each $\alpha\in\N^d$ and $[(x_\eps)_\eps]\in A$, $(\partial^\alpha u_\eps(x_\eps))_\eps\in\Mod_\C$. As $[(u_\eps(x_\eps))_\eps]\in u(A)\subseteq B$, also $(\partial^\alpha v_\eps(u_\eps(x_\eps)))_\eps\in\Mod_\C$. The moderateness-conditions follow inductively by the chain rule. Similarly, one sees that the definition does not depend on the representative of $v$. Independence of the representative of $u$: the estimates for $0$-th order derivatives follow as in \cite[Prop.~1.2.6]{GKOS} by corollary \ref{cor_EModA_internal} (applied to a singleton). Since $\tGen(A), \tGen(B)$ are closed under partial derivatives, the chain rule reduces the estimates for the higher order derivatives to the $0$-th order ones.
\end{proof}

\begin{prop}\label{prop_invertible}
Let $\emptyset \ne A\subseteq\GenR^d$ be internal and sharply bounded. Let $u=[(u_\eps)_\eps]\in\tGen(A)$. The following are equivalent for a sharply bounded representative $(A_{\eps})_\eps$ of $A$:
\begin{enumerate}
\item there exists $v\in\tGen(A)$ such that $uv=1$
\item for each $\tilde x\in A$, $u(\tilde x)$ is invertible in $\GenC$
\item $(\exists \eps_0 \in (0,1))$ $(\exists n\in\N)$ $(\forall\eps\le\eps_0)$ $(\inf_{x\in A_\eps + \eps^n} \abs{u_\eps(x)} \ge \eps^n)$.
\end{enumerate}
\end{prop}
\begin{proof}
$(1)\Rightarrow(2)$: for $\tilde x\in A$, $u(\tilde x)v(\tilde x)=1$ in $\GenC$.\\
$(2)\Rightarrow(3)$: supposing that the conclusion is not true, we find a decreasing sequence $(\eps_n)_{n\in\N}$ tending to $0$ and $x_{\eps_n}\in A_{\eps_n} + \eps_n^n$ and $\abs[]{u_{\eps_n}(x_{\eps_n})}<\eps_n^n$, for each $n\in\N$. Let $x_\eps\in A_{\eps}$, for small $\eps\notin\{\eps_n: n\in\N\}$. As $(A_\eps)_\eps$ is sharply bounded, $\tilde x:=[(x_\eps)_\eps]\in A$, but $u(\tilde x)$ is not invertible in $\GenC$ by \cite[Thm.\ 1.2.38]{GKOS}.\\
$(3)\Rightarrow(1)$: using a cut-off function, we find $v_\eps\in\Cnt[\infty](\R^d)$ with $v_\eps(x)=\inv{u_\eps(x)}$, for $x\in A_{\eps} + \eps^{n+1}$ and $\eps\le\eps_0$. Since each $\partial^\alpha v_\eps(x)$ is a linear combination (with coefficients indep.\ of $\eps$) of $\prod_{\beta}\partial^\beta u_\eps(x)/u_\eps^{\abs\alpha + 1}(x)$ (finite products) for $x\in A_\eps + \eps^{n+1}$ and $\eps\le\eps_0$, 
$v:=[(v_\eps)_\eps]\in\tGen(A)$. As $u_\eps(x) v_\eps(x) - 1 = 0$, for each $x\in A_\eps+\eps^{n+1}$ and $\eps\le\eps_0$, we have $uv=1$ in $\tGen(A)$.
\end{proof}

\begin{cor}\label{cor_invertible}
Let $\Omega = \bigcup_{n\in\N} A_n\ne\emptyset$, where $(A_n)_{n\in\N}$ is an increasing sequence of internal subsets of $\GenR^d$ such that $A_{n+1}$ is a neighbourhood of $A_n$, for each $n$. Let $u\in\tGen(\Omega)$. The following are equivalent:
\begin{enumerate}
\item there exists $v\in\tGen(\Omega)$ such that $uv=1$
\item for each $\tilde x\in \Omega$, $u(\tilde x)$ is invertible in $\GenC$
\item for each $m\in\N$, $\restr{u}{A_m}\in\tGen(A_m)$ has a multiplicative inverse.
\end{enumerate}
\end{cor}
\begin{proof}
$(1)\implies(3)$: by restriction.\\
$(3)\Rightarrow (1)$: let $v_m\in \tGen(A_m)$ such that $\restr{u}{A_m} v_m = 1$ in $\tGen(A_m)$, for each $m$. Then $v_m \cdot \restr{u}{A_m} \cdot \restr{v_{m+1}}{A_m} = \restr{v_{m+1}}{A_m} = v_m$, for each $m$. By proposition \ref{prop_sheaf}, there exists a unique $v\in\tGen(\Omega)$ with $\restr{v}{A_m}=v_m$, for each $m\in\N$. In particular, $u(\tilde x)v(\tilde x)=1$, for each $\tilde x\in \Omega$. Since $\Omega$ is open, $uv=1$ by proposition \ref{prop_GenA_pointvalues}.\\
Since $(1)\Leftrightarrow(3)$, property $(3)$ is independent of the choice of $A_m$ with $\Omega = \bigcup_m A_m$. As $\Omega = \bigcup_m (A_m\cap B(0,\caninf^m))$, $(2)\Leftrightarrow (3)$ follows by proposition \ref{prop_invertible}.
\end{proof}

\section{Analyticity on $\tGen(\GenC)$}
\begin{df}
Let $A$ be an open subset of $\GenC$. We let $\tGen_H(A)$ be the differential algebra consisting of those $u\in\tGen(A)$ with $\bar \partial u = \frac{1}{2}(\partial_x + i \partial_y) u = 0$. Let $A\subseteq B\subseteq\GenC$. We say that $u\in\tGen(B)$ is holomorphic in $A$ iff $\restr{u}{A}\in\tGen_H(A)$, i.e, iff $\bar\partial u (\tilde z) = 0$, for each $\tilde z\in A$.\\
For $u\in\tGen_H(A)$ and $\tilde z\in A$, we write $u'(\tilde z) = \partial_x u (\tilde z)=-i\partial_y u(\tilde z)$. Iterated derivatives are denoted by $D^k$ ($k\in\N$).
\end{df}
Clearly, every polynomial with coefficients in $\GenC$ (i.e., every element of $\GenC[z]$) belongs to $\tGen_H(\GenC)$.

\begin{lemma}\label{lemma_tGen_H_by_representatives}
For each $\eps\in(0,1)$, let $\Omega_\eps\subseteq\R^d$ be open, let $N\in\N$ and let $u_\eps$: $\Omega_\eps\to B(0,\eps^{-N})\subseteq\C$ be holomorphic. Let $B\subseteq \bigcup_{m\in\N} [(\{z\in\C: d(z,\C\setminus\Omega_\eps)\ge\eps^m\})_\eps]$ be open. Then $u:=[(u_\eps)_\eps]\in\tGen_H(B)$.
\end{lemma}
\begin{proof}
By lemma \ref{lemma_tGen_by_representatives}, it is sufficient to show that for each $m,k\in\N$, there exists $N\in\N$ such that $\sup_{d(z,\C\setminus\Omega_\eps)\ge\eps^m, \abs{z}\le\eps^{-m}} \abs{D^k u_\eps(z)}\le\eps^{-N}$, for small $\eps$. This follows by the Cauchy estimate $\abs{D^k u_\eps(z)}\le k! \eps^{-(m+1)k}\sup_{\partial B(z,\eps^{m+1})}\abs{u_\eps(z)}$ for $z\in\C$ with $d(z,\C\setminus\Omega_\eps)\ge\eps^m$.
\end{proof}
\begin{rem}
As in lemma \ref{lemma_tGen_by_representatives}, it is sufficient that $u_\eps$ satisfy: for each $m\in\N$, there exists $N\in\N$ such that $u_\eps$: $\{z\in\C: d(z,\C\setminus \Omega_\eps)\ge\eps^m$ and $\abs z\le \eps^{-m}\} \to B(0,\eps^{-N})$ is holomorphic.
\end{rem}

\begin{lemma}
Let $\Omega\subseteq\GenC$ be open. Let $u\in\tGen_H(\Omega)$.
\begin{enumerate}
\item If $1/u\in \tGen(\Omega)$ exists, $1/u\in \tGen_H(\Omega)$.
\item Let $\Omega = \bigcup_{n\in\N} A_n$, where $(A_n)_{n\in\N}$ is an increasing sequence of internal subsets of $\GenR^d$ such that $A_{n+1}$ is a neighbourhood of $A_n$, for each $n$. If $u(\tilde z)$ is invertible, for each $\tilde z\in \Omega$, then $1/u\in\tGen_H(\Omega)$.
\end{enumerate}
\end{lemma}
\begin{proof}
(1) If $uv=1$ and $\bar\partial u = 0$, then $0 = \bar\partial(uv) = u \cdot \bar\partial v$.\\
(2) By corollary \ref{cor_invertible} and part~1.
\end{proof}

\begin{lemma}\label{lemma_composition_analytic}
Let $A\subseteq \GenC$ and $u=[(u_\eps)_\eps]\in\tGen_H(A)$. Let $u(A)=\{u(\tilde z): \tilde z\in A\}\subseteq B$. Let $v=[(v_\eps)_\eps]\in\tGen_H(B)$. Then $v\comp u:=[(v_\eps\comp u_\eps)_\eps]\in\tGen_H(A)$.
\end{lemma}
\begin{proof}
By propositon \ref{prop_composition}, $v\comp u\in\tGen(A)$ and for $\tilde z\in A$, $\bar\partial u(\tilde z) = \bar\partial v(u(\tilde z))=0$. Hence $\bar\partial (v\comp u)(\tilde z) = \partial_x v(u(\tilde z))\bar\partial \Re u(\tilde z) + \partial_y v(u(\tilde z))\bar\partial \Im u(\tilde z) = v'(u(\tilde z)) \bar\partial u(\tilde z) = 0$, for each $\tilde z\in A$.
\end{proof}

\begin{df}\label{df_path}
Let $\Cnt[1]_{pw}([0,1])$ be the space of those $u\in\Cnt[0]([0,1])$ that are piecewise $\Cnt[1]$, provided with the norm $\max\{\sup_{x\in[0,1]}\abs{u(x)},\sup_{x\in[0,1]\text{ a.e.}}\abs{u'(x)}\}$.
We call a \defstyle{path} in $\GenC$ an element of $\Gen_{\Cnt[1]_{pw}([0,1])}$.
\end{df}
So for a representative $(\gamma_\eps)_\eps$ of a path $\gamma$, we have $\gamma_\eps\in\Cnt[1]_{pw}([0,1])$, $\forall\eps$, and the nets \[\big(\sup\nolimits_{t\in[0,1]}\abs{\gamma_\eps(t)}\big)_\eps\quad \text{and}\quad \big(\sup\nolimits_{t\in[0,1]\text{ a.e.}}\abs{\gamma'_\eps(t)}\big)_\eps\]
are both moderate. Since
\[
\abs{\gamma_\eps(t_\eps)-\gamma_\eps(t'_\eps)} =\abs[\bigg]{\int_{t_\eps}^{t'_\eps} \gamma_\eps'(t)\,dt}\le \abs{t_\eps-t'_\eps}\sup_{t\in [0,1]\text{ a.e}}\abs{\gamma_\eps'(t)},
\]
generalized point values $\gamma(\tilde t)$ are well-defined, for each $\tilde t\in[0,1]\sptilde$. On the other hand, if $\gamma(\tilde t)=\tilde\gamma(\tilde t)$, for each $\tilde t\in[0,1]\sptilde$, then the paths $\gamma,\tilde\gamma$ need not be equal (e.g., they can have different curve length). If $A\subseteq\GenC$ and $\gamma(\tilde t)\in A$, for each $\tilde t\in [0,1]\sptilde$, we call $\gamma$ a path in $A$.

\begin{prop}\label{prop_path_integral_well_def}
Let $A\subseteq\GenC$, $u\in\tGen(A)$ and $\gamma$ a path in $A$. Then \[\int_\gamma u(z)\,dz:=\Big[\Big(\int_{\gamma_\eps} u_\eps(z)\,dz\Big)\Big]\in\GenC\]
is well-defined (independent of representatives of $u$ and $\gamma$).
Moreover, if $\gamma$, $\tilde\gamma$ are paths in $A$ such that $\gamma(\tilde t)=\tilde\gamma(\tilde t)$, for each $\tilde t\in[0,1]\sptilde$, then $\int_{\gamma} u(z)\,dz = \int_{\tilde\gamma} u(z)\,dz$.
\end{prop}
\begin{proof}
Let $\gamma=[(\gamma_\eps)_\eps]$. Since $u\in\tGen(A)$ and
$[(\gamma_\eps([0,1]))_\eps]=\{\gamma(\tilde t): \tilde t\in [0,1]\sptilde\}\subseteq A$, corollary \ref{cor_EModA_internal} implies that $\sup_{z\in \gamma_\eps([0,1])}\abs{u_\eps(z)}\in\Mod_\R$. Hence $(\int_{\gamma_\eps}u_\eps(z)\,dz)_\eps\in\Mod_\C$. Independence of the representative of $u$ follows similarly by proposition \ref{prop_NullA_internal}.

To prove independence of the representative of $\gamma$, we use an argument similar to Green's theorem. 
More generally, let $\gamma=[(\gamma_\eps)_\eps]$ and $\tilde\gamma=[(\tilde\gamma_\eps)_\eps]$ as in the statement of the theorem. Let $\eps\in(0,1)$. Let $[a_\eps,b_\eps]\subseteq[0,1]$ such that $\gamma_\eps$, $\tilde\gamma_\eps\in\Cnt[1]([a_\eps,b_\eps])$. Consider the homotopy $H_\eps$: $[0,1]\times[0,1]\to\C$: $H_\eps(t,s)=\gamma_\eps(t)+ s (\tilde \gamma_\eps(t)-\gamma_\eps(t))$.
As 
\begin{multline*}
\partial_t \big((u_\eps\comp H_\eps)\cdot\partial_s H_\eps\big) - \partial_s \big((u_\eps\comp H_\eps)\cdot\partial_t H_\eps\big)\\
=i\big(((\partial_x + i\partial_y)u_\eps)\comp H_\eps\big)\cdot (\partial_t H_{1,\eps}\partial_s H_{2,\eps} - \partial_s H_{1,\eps}\partial_t H_{2,\eps}),
\end{multline*}
integration over $[0,1]^2$ (via integration over different $[a_\eps,b_\eps]\times[0,1]$ and summation), yields
\begin{multline*}
\int_{\tilde\gamma_\eps} u_\eps(z)\,dz - \int_{\gamma_\eps} u_\eps(z)\,dz = \int_0^1 [(u_\eps\comp H_\eps)(t,s) \partial_s H_\eps(t,s)]_{t=0}^{t=1}\,ds\\
- i\iint_{[0,1]^2} \big((\partial_x + i\partial_y) u_\eps\big)(H_\eps(t,s))\cdot \big(\partial_t H_{1,\eps} \partial_s H_{2,\eps}
- \partial_s H_{1,\eps}\partial_t H_{2,\eps}\big)(t,s)\,dt\,ds.
\end{multline*}
Since $u\in\tGen(A)$ and
\[
[(H_\eps([0,1]^2))_\eps]=\{H(\tilde t, \tilde s): \tilde t, \tilde s\in [0,1]\sptilde\}=\{\gamma(\tilde t): \tilde t\in [0,1]\sptilde\}\subseteq A,
\]
corollary \ref{cor_EModA_internal} implies that $\sup_{z\in H_\eps([0,1]^2)}\abs{\partial^\alpha u_\eps(z)}\in\Mod_\R$, for each $\alpha\in\N^d$.
The moderateness of $\partial_t H_\eps$ and the negligibility of $\partial_s H_\eps(t,s)= \tilde\gamma_{\eps}(t)-\gamma_{\eps}(t)$ then yield the required negligibility.
\end{proof}
Since $\Gen(\Omega)=\tGen(\widetilde\Omega_c)$ and a c-bounded (cf.\ \cite[Def. 1.2.7]{GKOS}) path in $\Omega$ is a path in $\widetilde\Omega_c$, we immediately obtain the following corollary.
\begin{cor}
If $u\in\Gen(\Omega)$ and $\gamma$ is a c-bounded path in $\Omega$, 
then $\int_\gamma u(z)\,dz$ is well-defined.
\end{cor}

\begin{df}
Let $\Cnt[1]_{pw}([0,1]^2)$ be the space of those $u\in\Cnt[0]([0,1]^2)$ for which there exists a partition $(a_j)_{j=0}^n$ of $[0,1]$ with $u\in\Cnt[1]([a_{i-1},a_i]\times[a_{j-1},a_j])$, for $i,j=1$, \dots, $n$, provided with the norm $\max\{\sup_{x\in[0,1]^2}\abs{u(x)}, \sup_{x\in[0,1]^2\text{ a.e.}}\abs{\nabla u(x)}\}$. As for paths, one sees that for $u\in\Gen_{\Cnt[1]_{pw}([0,1]^2)}$, generalized point values are well defined (for each $(\tilde t,\tilde s)\in([0,1]\sptilde)^2$).\\
Let $\gamma$, 
$\tilde\gamma$ 
be two closed (i.e., $\gamma(0)=\gamma(1)$) paths in $\GenC$. We call $H\in\Gen_{\Cnt[1]_{pw}([0,1]^2)}$ a \defstyle{homotopy} between $\gamma$ and $\tilde\gamma$ if $H(\tilde t,0)=\gamma(\tilde t)$, $H(\tilde t,1)=\tilde\gamma(t)$, $\forall\tilde t\in[0,1]\sptilde$ and $H(0,\tilde s)=H(1,\tilde s)$, $\forall\tilde s\in[0,1]\sptilde$.
$H$ is called a homotopy \defstyle{in $A\subseteq\GenC$} if $H(\tilde t,\tilde s)\in A$, for each $\tilde t,\tilde s\in [0,1]\sptilde$.\\
If each two closed paths in $A$ are homotopic in $A$, we call $A$ \defstyle{simply connected}.
\end{df}
\begin{ex}
Let $A\subseteq\GenC$ be (pointwise) convex, i.e., for each $\tilde z_1,\tilde z_2\in A$ and $\tilde t\in [0,1]\sptilde$, $t z_1 + (1-t) z_2\in A$. Then $A$ is simply connected.
\end{ex}
\begin{proof}
Let $\gamma=[(\gamma_\eps)_\eps]$, $\tilde\gamma=[(\tilde\gamma_\eps)_\eps]$ are two closed paths in $A$, the homotopy defined by $H_\eps(t,s)=\gamma_\eps(t) + s (\tilde \gamma_\eps(t) - \gamma_\eps(t))$ is a homotopy in $A$ between $\gamma$ and $\tilde\gamma$.
\end{proof}

\begin{prop}\label{prop_Cauchy_thm}
Let $A\subseteq\GenC$ be open, $u\in\tGen_H(A)$ and $\gamma$, $\tilde \gamma$ two closed paths, homotopic in $A$.
Then 
\[\oint_{\gamma}u(z)\,dz=\oint_{\tilde\gamma}u(z)\,dz.\]
\end{prop}
\begin{proof}
Let $H=[(H_\eps)_\eps]$ be the given homotopy in $A$ between $\gamma$ and $\tilde \gamma$. As $[(H_\eps([0,1]^2))_\eps]\subseteq A$ and $\bar\partial u=0$ in $\tGen(A)$, proposition \ref{prop_NullA_internal} implies that $\sup_{z\in H_\eps([0,1]^2)}\abs{\bar\partial u_\eps(z)}$ is negligible. As in proposition \ref{prop_path_integral_well_def} (now using the given homotopies $H_\eps$ and integration for fixed $\eps$ over each $[a_{i-1,\eps},a_{i,\eps}]\times[a_{j-1,\eps},a_{j,\eps}]$, and summation), we find $(\nu_\eps)_\eps\in\Null_\C$ such that for each $\eps$,
\[\int_{\tilde\Gamma_{\eps}} u_\eps(z)\,dz - \int_{\Gamma_{\eps}} u_\eps(z)\,dz = \int_{\tilde\delta_\eps} u_\eps(z)\,dz - \int_{\delta_\eps} u_\eps(z)\,dz + \nu_\eps,\]
where $\Gamma_\eps(t)=H_\eps(t,0)$, $\tilde\Gamma_\eps(t)=H_\eps(t,1)$, $\delta_\eps(s)=H_\eps(0,s)$ and $\tilde\delta_\eps(s)=H_\eps(1,s)$, $\forall t,s\in [0,1]$ and $\eps\in (0,1)$. Since $\Gamma:=[(\Gamma_\eps)_\eps]$, $\tilde\Gamma:=[(\tilde\Gamma_\eps)_\eps]$, $\delta:=[(\delta_\eps)_\eps]$ and $\tilde\delta:=[(\tilde\delta_\eps)_\eps]$ are paths in $A$ and $\Gamma(\tilde t) = \gamma(\tilde t)$, $\tilde\Gamma(\tilde t) = \tilde\gamma(\tilde t)$, $\delta(\tilde s)=\tilde\delta(\tilde s)$, $\forall \tilde t,\tilde s\in [0,1]\sptilde$, the statement follows by applying proposition \ref{prop_path_integral_well_def}.
\end{proof}
\begin{cor}
Let $A\subseteq\GenC$ be open and simply connected, $u\in\tGen_H(A)$ and $\gamma$ a closed path in $A$. Then
\[\oint_{\gamma}u(z)\,dz=0.\]
\end{cor}
\begin{proof}
Since $A$ is simply connected, there exists a homotopy between $\gamma$ and a constant path. The result follows by the previous proposition.
\end{proof}

Let $\tilde x\in\GenR$. We write $\tilde x\gg 0$ iff $\tilde x\ge 0$ and $\tilde x$ invertible in $\GenR$ (i.e., $\abs{\tilde x}\ge \caninf^m$, for some $m\in\N$ \cite[Thm.\ 1.2.38]{GKOS}). Similarly, $\tilde x\gg \tilde y$ iff $\tilde y\ll \tilde x$ iff $\tilde x - \tilde y\gg 0$.
\begin{prop}\label{prop_Cauchy_formula}
Let $r\in\GenR$, $r\gg 0$. Let $\tilde a\in\GenC$. Let $A\subseteq\GenC$ be open with $\{\tilde \zeta\in\GenC: \abs[]{\tilde \zeta - \tilde a}\le r\}\subseteq A$. For $u\in\tGen_H(A)$, $k\in\N$ and $\gamma=\partial B(\tilde a,r)$ with positive orientation and $\tilde z\in\GenC$ with $\abs{\tilde z - \tilde a}\ll r$,
\[D^k u(\tilde z) = \frac{k!}{2\pi i}\oint_\gamma \frac{u(\zeta)}{(\zeta - \tilde z)^{k+1}}\,d\zeta.\]
\end{prop}
\begin{proof}
Let $\tilde z=[(z_\eps)_\eps]$.
Let $M\in\N$ such that $\abs{\tilde z - \tilde a}\le r - \caninf^M$. For $m>M$, let $A_m=\{\tilde \zeta\in\GenC: \abs[]{\tilde \zeta - \tilde z}\ge \caninf^m, \abs[]{\tilde \zeta - \tilde a}\le r \}$. Since $v(\zeta):=\frac{u(\zeta)}{(\zeta - \tilde z)^{k+1}}\in \Gen_H(A_m)$, for each $m\in\N$, we may perform the integration over $\gamma_m:=\partial B(\tilde z,\caninf^m)$ instead of $\gamma$ (for any $m\in\N$) by proposition \ref{prop_Cauchy_thm}. Let $k=0$. Since $u\in\EMod(\{\tilde z\})$, as in proposition \ref{prop_GenA_pointvalues}, there exists $N\in\N$ such that for sufficiently large $m$, $\sup_{\abs{z-z_\eps}\le\eps^m}\abs{u_\eps(z)-u_\eps(z_\eps)}\le\eps^{-N}\sup_{\abs{z-z_\eps}\le \eps^m}{\abs{z-z_\eps}}\le \eps^{m-N}$.
Hence
\[
\abs{\frac{1}{2\pi i}\oint_{\gamma_m} \frac{u(\zeta)}{\zeta - \tilde z}\,d\zeta - u(\tilde z)}
= \abs{\frac{1}{2\pi i}\oint_{\gamma_m} \frac{u(\zeta) - u(\tilde z)}{\zeta - \tilde z}\,d\zeta}
\le \caninf^{m-N},
\]
for each sufficiently large $m$. Let $u=[(u_\eps)_\eps]$ and $\gamma=[(\gamma_\eps)_\eps]$. By differentiation under the integral sign, $\big(\int_{\gamma_\eps} \frac{u_\eps(\zeta)}{\zeta - z}\,d\zeta\big)_\eps\in\EMod(\{\tilde \zeta\in\GenC: \abs[]{\tilde \zeta - \tilde a}\ll r\})$, hence $u(z) = \frac{1}{2\pi i}\int_\gamma \frac{u(\zeta)}{\zeta - z}\,d\zeta$ in $\tGen(\{\tilde \zeta\in\GenC: \abs[]{\tilde \zeta - \tilde a}\ll r\})$ by proposition \ref{prop_GenA_pointvalues}(3). By definition of $\tGen(\{\tilde z\})$, this implies that all partial derivatives in the point $\tilde z$ are equal.
\end{proof}

\begin{cor}\label{cor_Cauchy_formula}
Let $r\in\GenR$, $r\gg 0$. Let $\tilde a\in\GenC$. Let $A\subseteq\GenC$ be open with $\{\tilde \zeta\in\GenC: \abs[]{\tilde \zeta - \tilde a}\le r\}\subseteq A$. Let $u\in\tGen_H(A)$ and let $\tilde z\in\GenC$ with $\abs{\tilde z - \tilde a}\ll r$. Then for each $k\in\N$, $\abs{D^k u(\tilde z)}\le k! r^{-k}\max_{\abs[]{\tilde \zeta -\tilde a} = r}\abs[]{u(\tilde \zeta)}$.
\end{cor}
\begin{proof}
By the previous proposition, since for $\tilde a=[(a_\eps)_\eps]$, $r=[(r_\eps)_\eps]$ and $u=[(u_\eps)_\eps]$,
\[\max_{\abs[]{\tilde \zeta -\tilde a} = r}\abs[]{u(\tilde \zeta)}=\big[\big(\max_{\abs{\zeta - a_\eps} = r_\eps}\abs{u_\eps(\zeta)}\big)_\eps\big].\]
\end{proof}

\begin{prop}[Liouville's theorem]\label{prop_Liouville}
If $u\in\tGen_H(\GenC)$ is bounded (i.e., there exists $C\in\GenR$ such that $\abs{u(\tilde z)}\le C$, for each $\tilde z\in\GenC$), then $u$ is a generalized constant.
\end{prop}
\begin{proof}
By corollary \ref{cor_Cauchy_formula}, $u'(\tilde z)=0$, for each $\tilde z\in\GenC$. Hence $u'=0$ in $\tGen(\GenC)$. As in \cite[Prop.~1.2.35]{GKOS}, this implies that $u$ is a generalized constant.
\end{proof}

\begin{prop}
If $u\in\tGen_H(\GenC)$ is of polynomial growth (i.e., there exist $C\in\GenR$ and $m\in\N$ such that $\abs{u(\tilde z)}\le C + C\abs{\tilde z}^m$, for each $\tilde z\in\GenC$),
then $u\in\GenC[z]$.
\end{prop}
\begin{proof}
Let $\tilde z\in\GenC$ and $n\in\N$. Then by corollary \ref{cor_Cauchy_formula},
\begin{multline*}
\abs{D^{m+1} u(\tilde z)} \le (m+1)!\caninf^{n(m+1)}\max_{\abs[]{\tilde \zeta - \tilde z}=\caninf^{-n}}\abs[]{u(\tilde \zeta)}\\
\le (m+1)!\caninf^{n(m+1)}(C + C(\abs{\tilde z} + \caninf^{-n})^m)\le \caninf^{n - M},
\end{multline*}
for some $M\in\N$ only depending on $C$, $m$ and $\tilde z$. As $n\in\N$ arbitrary, $D^{m+1} u (\tilde z)= 0$. As $\tilde z\in\GenC$ arbitrary, $D^{m+1} u = 0$. Hence $u\in\GenC[z]$ (as in $\Gen(\C)$, $u'=0$ implies that $u$ is a generalized constant, cf.\ \cite[Prop.~1.2.35]{GKOS}).
\end{proof}

\begin{lemma}\label{lemma_convergence_radius}
Let $a_n\in\GenC$, for each $n\in\N$. Then the sum $\sum_{n=0}^\infty a_n \tilde z^n$ converges for each $\tilde z\in\GenC$ with $\sharpnorm{\tilde z}< R$ and does not converge for each invertible $\tilde z\in\GenC$ with $\sharpnorm{\inv{\tilde z}}< 1/R$, where $R=1/\limsup_{n\to\infty}\sqrt[n]{\sharpnorm{a_n}}\in[0,+\infty]$. Moreover, convergence is uniform over each ball $\{\tilde z\in\GenC: \sharpnorm{\tilde z} \le r\}$, where $r<R$.
\end{lemma}
\begin{proof}
Let $r<r'<R$. Let $\tilde z\in\GenC$ with $\sharpnorm{\tilde z}\le r$. By the ultrapseudonorm property of the sharp norm, $\sum_{n=0}^\infty a_n \tilde z^n$ converges iff $\lim_{n\to\infty} a_n \tilde z^n = 0$ (in the sharp topology). By definition of $R$, $\sqrt[n]{\sharpnorm{a_n}}\le 1/r'$, as soon as $n$ is large enough. Hence $\sharpnorm{a_n \tilde z^n}\le\sharpnorm{a_n}\sharpnorm{\tilde z}^n\le (r/r')^n \to 0$.\\
Let $\tilde z\in\GenC$ invertible and $\sharpnorm{\inv{\tilde z}}\le 1/r< 1/R$. By definition of $R$, there are infinitely many $n\in\N$ such that $\sqrt[n]{\sharpnorm{a_n}}\ge 1/r$. Then $\sharpnorm{a_n \tilde z^n}\ge \sharpnorm{a_n \tilde z^n} r^n \sharpnorm{\inv{\tilde z}}^n\ge \sharpnorm{a_n} r^n\ge 1\not\to 0$ as $n\to\infty$.
\end{proof}
We call $R$ the \defstyle{convergence radius} of the power series.

\begin{lemma}\label{lemma_conv_radius_internal}
Let $\emptyset \ne A\subseteq\GenR^d$ be internal and sharply bounded. Let $(A_\eps)_\eps$ be a sharply bounded representative of $A$. Let $u\in\tGen(A)$. Then $\limsup_{\abs\alpha\to\infty}\sqrt[\abs\alpha]{\sharpnorm{\partial^\alpha u(\tilde x)}}\le 1$, $\forall \tilde x\in A$ iff
\[(\forall c \in\R^+) (\exists N\in\N) (\forall\alpha\in\N^d, \abs{\alpha}\ge N) \big(\sup_{x\in A_\eps}\abs{\partial^\alpha u_\eps(x)}\le \eps^{-c\abs{\alpha}}, \text{ for small }\eps\big).\]
\end{lemma}
\begin{proof}
$\Rightarrow$: Supposing the conclusion is not true, we find $c\in\R^+$, $\alpha_n\in\N^d$ with $\abs{\alpha_n}\ge n$ (for each $n\in\N$), $\eps_{n,m}\in (0,1/m)$ (for each $n,m\in\N$) (by enumerating the countable family $(\eps_{n,m})_{n,m}$, we can successively choose the $\eps_{n,m}$ in such a way that they are all different) and $x_{\eps_{n,m}}\in  A_{\eps_{n,m}}$ with $\abs{\partial^{\alpha_n}u_{\eps_{n,m}}(x_{\eps_{n,m}})} > \eps_{n,m}^{-c\abs{\alpha_n}}$, $\forall n,m\in\N$. Choose $x_\eps \in A_\eps$ arbitrary, if $\eps\notin\{\eps_{n,m}: n,m\in\N\}$ is sufficiently small ($A_\eps\ne\emptyset$ for small $\eps$ since $A\ne \emptyset$). Then $\tilde x=[(x_\eps)_\eps]\in A$ (moderateness is guaranteed since $(A_\eps)_\eps$ is sharply bounded). By hypothesis, there exists $N\in\N$ such that for each $\alpha\in\N^d$ with $\abs{\alpha}\ge N$, $\abs{\partial^\alpha u(\tilde x)}\le\caninf^{-c\abs{\alpha}}$. This contradicts the fact that for a fixed $n>N$, $\abs{\alpha_n}\ge N$, $\lim_{m\to\infty}\eps_{n,m}=0$ and $\abs{\partial^{\alpha_n} u_{\eps_{n,m}}(x_{\eps_{n,m}})}> \eps_{n,m}^{-c\abs{\alpha_n}}$, $\forall m\in\N$.\\
$\Leftarrow$: clear.
\end{proof}

\begin{prop}\label{prop_Geninfty_implies_GenH}
Let $\tilde z_0\in A\subseteq\{\tilde z\in\GenC: \sharpnorm{\tilde z - \tilde z_0} < 1\}$ and let $A$ be open and star-shaped around $\tilde z_0$ (i.e., for each $\tilde z\in A$ and $\tilde t\in[0,1]\sptilde$, $\tilde t\tilde z + (1-\tilde t)\tilde z_0\in A$). Let $u\in\tGen(A)$ with $\limsup_{\abs\alpha\to\infty} \sqrt[\abs\alpha]{\sharpnorm{\partial^\alpha u(\tilde z)}}\le 1$, $\forall\tilde z\in A$ and $\bar\partial u (\tilde z_0) = 0$. Then $u(\tilde z)=\sum_{n=0}^\infty \frac{D^n(\tilde z_0)}{n!}(\tilde z - \tilde z_0)^n$, for each $\tilde z\in A$ and $u$ has a representative consisting of polynomials in $\C[z]$ (in particular, $u\in\tGen_H(A)$).
\end{prop}
\begin{proof}
We may suppose that $\tilde z_0 =0$ (consider the translated generalized function $v(z)=u(z+\tilde z_0)$). Let $a\in\R^+$. Then $\sharpnorm{D^n u(0)}\le e^{na}$, for large $n$. Hence we can find $a_n\in\R$ with $a_n\to 0$ and representatives $(c_{n,\eps})_\eps$ of $D^n u(0)$ with $\abs{c_{n,\eps}}\le\eps^{-na_n}$, $\forall\eps$. Let $w_\eps(z)=\sum_{n=0}^{m_\eps} \frac{c_{n,\eps}}{n!} z^n$, for each $\eps\in (0,1)$, where $\lim_{\eps\to 0}m_\eps=\infty$. Let $a\in\R^+$. As there exists $N\in\N$ such that for each $n\ge N$, $a_n\le a/2$, we find for each $k\in\N$,
\begin{multline*}
\sup_{\abs z \le \eps^a}\abs{D^k w_\eps(z)}\le \sup_{\abs z\le \eps^a}\sum_{n=k}^{m_\eps} \frac{\eps^{-n a_n}\abs[]z^{n-k}}{(n-k)!}
\le \eps^{-ka}\sum_{n < N} \eps^{(a-a_n)n} + \eps^{-ka}\sum_{n=N}^{m_\eps} \eps^{an/2}\\
\le \eps^{-ka}\sum_{n < N} \eps^{(a-a_n)n} + \frac{\eps^{-ka}}{1 - \eps^{a/2}} \le \eps^{-ka-M},
\end{multline*}
for some $M\in\N$ not depending on $k$ and for small $\eps$. As $a\in\R^+$ is arbitrary, $w\in\tGen(A)$ and $\limsup_{\abs\alpha\to\infty} \sqrt[\abs\alpha]{\sharpnorm{\partial^\alpha w(\tilde z)}}\le 1$, $\forall\tilde z\in A$. Further, $D^k w_\eps(0)=c_{k,\eps}$, for small $\eps$, hence $D^k w(0)=D^k u(0)$, for each $k\in\N$. Since $(\partial_x + i \partial_y) u(0)=(\partial_x + i \partial_y) w(0)=0$, also $\partial^\alpha w(0) = \partial^\alpha u(0)$, for each $\alpha\in \N^2$.\\
Let $f=u-w \in\tGen(A)$. Then $\partial^\alpha f(0)=0$, $\forall\alpha\in\N^2$ and $\limsup_{\abs\alpha\to\infty} \sqrt[\abs\alpha]{\sharpnorm{\partial^\alpha f(\tilde z)}}\le 1$, $\forall\tilde z\in A$. Let $\tilde z=[(z_\eps)_\eps]\in A$. Then there exists $a\in\R^+$ such that $\abs{z_\eps}\le\eps^a$, for small $\eps$. Since $[(\{tz_\eps:t\in [0,1])_\eps]=\{\tilde t \tilde z: \tilde t\in [0,1]\sptilde\}\subseteq A$ is internal and sharply bounded, lemma \ref{lemma_conv_radius_internal} implies that for each $\alpha\in\N^d$, $\sup_{t\in[0,1]}\abs{\partial^\alpha f_\eps(t z_\eps)}\le\eps^{-a\abs\alpha/2}$, for small $\eps$.
By the Taylor expansion up to order $m$ (in two real variables),
\[
\abs{f_\eps(z_\eps)}\le\nu_\eps + \abs {z_\eps}^{m+1}\sum_{\abs\alpha=m+1}\sup_{t\in [0,1]}\abs{\partial^{\alpha} f_\eps(t z_\eps)} \le\eps^{a(m+1)/2-1},
\]
for small $\eps$ and for some $(\nu_\eps)_\eps\in\Null_\R$. As $m\in\N$ arbitrary and $a>0$, $f(\tilde z)=0$, i.e., $u(\tilde z)= w(\tilde z)$. By proposition \ref{prop_GenA_pointvalues}, $(w_\eps)_\eps$ is a representative of $u$.\\
By lemma \ref{lemma_convergence_radius}, the convergence radius of the power series $\sum_n \frac{D^n u(0)}{n!} z^n$ is at least equal to $1$, hence the Taylor expansion of $u$ around $\tilde z_0$ converges for each $\tilde z\in A$. 
Fix $\tilde z\in A$. Let $a_m := \sum_{n=0}^m \frac{D^n u(0)}{n!} \tilde z^n\in\GenC$, $\forall m\in\N$. By the convergence of $(a_m)_{m\in\N}$ (in the sharp topology), one easily shows that there exist $\eps_m\in (0,1)$ with $\lim_{m\to\infty} \eps_m=0$ such that $b_\eps:= a_{m,\eps}$, for $\eps_{m+1}<\eps\le\eps_m$ defines a representative $(b_\eps)_\eps$ of $\lim_{m\to\infty} a_m$. If we choose well the representatives of $a_m$, $(b_\eps)_\eps=\big(\sum_{n=0}^{m_\eps} \frac{c_{n,\eps}}{n!} z^n\big)_\eps$ with $c_{n,\eps}$ and $m_\eps$ as before. Hence $\sum_{n=0}^\infty \frac{D^n u(0)}{n!} \tilde z^n= u(\tilde z)$.
\end{proof}
E.g., if $u\in\tGen^\infty(A)$, then $\limsup_{\abs\alpha\to\infty} \sqrt[\abs\alpha]{\sharpnorm{\partial^\alpha w(\tilde z)}}\le 1$, $\forall\tilde z\in A$.

\begin{prop}\label{prop_GenA_with_bounds_implies_GenH}
Let $r\in\R$, $r>0$. Let $\tilde z_0\in\GenC$. Let $\tilde z_0\in A\subseteq\{\tilde z\in\GenC: \sharpnorm{\tilde z - \tilde z_0} < r\}$ and let $A$ be open and star-shaped around $\tilde z_0$ (i.e., for each $\tilde z\in A$ and $\tilde t\in[0,1]\sptilde$, $\tilde t\tilde z + (1-\tilde t)\tilde z_0\in A$). Let $u\in\tGen(A)$. Suppose that for each $\tilde z\in A$, there exists $N\in\N$ such that for each $\alpha\in\N^2$, $\abs{\partial^\alpha u(\tilde z)}\le\caninf^{\abs\alpha \ln r - N}$ and let $\bar\partial u (\tilde z_0) = 0$. Then $u(\tilde z)=\sum_{n=0}^\infty \frac{D^n(\tilde z_0)}{n!}(\tilde z - \tilde z_0)^n$, for each $\tilde z\in A$ and $u$ has a representative consisting of polynomials in $\C[z]$ (in particular, $u\in\tGen_H(A)$).
\end{prop}
\begin{proof}
As in proposition \ref{prop_Geninfty_implies_GenH}, we may suppose that $\tilde z_0=0$. Let $u=[(u_\eps)_\eps]$ and let $v_\eps(z)=u_\eps(\eps^{-\ln r} z)$, for each $\eps$. Then $v=[(v_\eps)_\eps]\in \tGen(\caninf^{\ln r} A)$, with $\{0\}\in\caninf^{\ln r} A \subseteq \{\tilde z\in\GenC: \sharpnorm{\tilde z} < 1\}$ star-shaped around $0$. Further, $\partial^\alpha v(\tilde z)=\caninf^{-\abs\alpha \ln r}\partial^\alpha u(\tilde z)$, for each $\tilde z\in\caninf^{\ln r} A$ and $\alpha\in\N^2$. Hence $v\in\tGen^\infty(\caninf^{\ln r} A)$ and $\bar\partial v(0)=0$. The assertion follows by applying proposition \ref{prop_Geninfty_implies_GenH}.
\end{proof}

\begin{thm}\label{thm_holomorphic_charac}
Let $\tilde z_0\in\GenC$ and $r\in\R^+$. Let $A=\{\tilde z\in\GenC: \sharpnorm{\tilde z - \tilde z_0} < r\}$. The following are equivalent:
\begin{enumerate}
\item $u\in\tGen_H(A)$
\item $u\in\tGen(A)$, $\limsup_{\abs\alpha\to\infty}\sqrt[\abs\alpha]{\sharpnorm{\partial^\alpha u(\tilde z)}}\le 1/r$, $\forall \tilde z\in A$ and $\bar\partial u (\tilde z_0) = 0$
\item $u$: $A\to\GenC$: $u(\tilde z)=\sum_{n=0}^\infty a_n (\tilde z - \tilde z_0)^n$, for some $a_n\in\GenC$
\item $u\in\tGen(A)$ has a representative consisting of polynomials ($\in\C[z]$).
\end{enumerate}
\end{thm}
\begin{proof}
We may suppose that $\tilde z_0=0$ and $r=1$ (consider $w(\tilde z):=u(\tilde z_0 + \caninf^{-\ln r}\tilde z)$).\\
$(1)\Rightarrow(2)$: let $\tilde z\in A$ and $a\in\R$, $a>0$. Then $B(\tilde z, \caninf^a)\subseteq A$. By corollary \ref{cor_Cauchy_formula}, for each $k\in\N$, $\abs{D^k u(\tilde z)}\le k!\caninf^{-a k}\max_{\abs[]{\tilde \zeta - \tilde z}=\caninf^a}\abs[]{u(\tilde \zeta)}$. Hence $\limsup_{n\to\infty}\sqrt[n]{\sharpnorm{D^n u(\tilde z)}}\le e^{a}$, $\forall a>0$.\\
$(2)\Rightarrow(3)$: by proposition \ref{prop_Geninfty_implies_GenH}.\\
$(3)\Rightarrow(4)$: let $c\in\R$, $c>0$. By lemma \ref{lemma_convergence_radius}, $\limsup_{n\to\infty}\sqrt[n]{\sharpnorm{a_n}}\le 1$. Hence $\sharpnorm{a_n}\le e^{nc}$, for large $n$. Hence we can find $c_n\in\R$ with $c_n\to 0$ and representatives $(a_{n,\eps})_\eps$ of $a_n$ with $\abs{a_{n,\eps}}\le\eps^{-nc_n}$, $\forall\eps$. Let $v_\eps=\sum_{n=0}^{m_\eps} a_{n,\eps}z^n$, $\forall\eps\in(0,1)$, where $\lim_{\eps\to 0}m_\eps=\infty$. Let $c\in \R$, $c>0$. As there exists $N\in\N$ such that for each $n\ge N$, $c_n\le c/2$, we find for $k\in\N$,
\[\sup_{\abs{z}\le \eps^c}\abs{D^k v_\eps(z)}\le \sum_{n <N}\frac{n!\abs{a_{n,\eps}}\eps^{nc}}{(n-k)!}  + \sum_{n\ge N} \frac{n!\eps^{cn/2-kc}}{(n-k)!} \le \sum_{n <N}\frac{n!\abs{a_{n,\eps}}\eps^{nc}}{(n-k)!} + \frac{k!\eps^{-kc/2}}{(1- \eps^{c/2})^{k+1}},\]
for small $\eps$. As $c>0$ is arbitrary, $(v_\eps)_\eps\in\EMod(A)$ by corollary \ref{cor_EModA_internal}.\\
Let $v=[(v_\eps)_\eps]\in\tGen(A)$. Let $c\in\R$, $c>0$. Let $m\in\N$ sufficiently large. Similarly,
\[\sup_{\abs{z}\le \eps^c}\abs[\Big]{v_\eps(z)- \sum_{n\le m} a_{n,\eps} z^n}\le \sum_{n > m} \eps^{cn/2}\le \frac{\eps^{cm/2}}{1-\eps^{c/2}}\]
for small $\eps$. As $c>0$ is arbitrary, $v(\tilde z)= u(\tilde z)$, for each $\tilde z\in A$.\\
$(4)\Rightarrow(1)$: clear.
\end{proof}
The example $u(\tilde z)=\sum_{n=0}^\infty \caninf^{-\frac{n}{\ln n}} \tilde z^n$ shows that the conditions of the previous proposition may be fulfilled for $u\in\tGen(A)\setminus\tGen^\infty(A)$ (in the case $r=1$).

\begin{cor}
The following are equivalent:
\begin{enumerate}
\item $u\in\tGen_H(\GenC)$
\item $u\in\tGen(\GenC)$, $\limsup\limits_{\abs\alpha\to\infty}\sqrt[\abs\alpha]{\sharpnorm{\partial^\alpha u(\tilde z)}} = 0$, $\forall \tilde z\in \GenC$ and $\bar\partial u (\tilde z_0) = 0$, for some $\tilde z_0\in\GenC$
\item for each $\tilde z\in\GenC$, $u$: $\GenC\to\GenC$: $u(\tilde z)=\sum_{n=0}^\infty a_n\tilde z^n$, for some $a_n\in\GenC$
\item $u\in\tGen(\GenC)$ has a representative consisting of polynomials ($\in\C[z]$).
\end{enumerate}
\end{cor}
\begin{proof}
$(3)\implies(4)$: as in the proof of proposition \ref{thm_holomorphic_charac}, now with $c_n\to -\infty$.\\
The other implications follow directly from proposition \ref{thm_holomorphic_charac}.
\end{proof}

\begin{ex}
Let $u(\tilde z)=\sum_{n=0}^\infty\frac{\caninf^{n^2}}{n!}\tilde z^n$. Then $u\in \Gen_H(\GenC)\setminus\GenC[z]$ is the unique element of  $\Gen_H(\GenC)$ with $D^ku(0)=\caninf^{k^2}$, for each $k\in\N$.
\end{ex}

\begin{df}
Let $A\subseteq \GenC$. Then $\tilde z_0$ is a \defstyle{strict accumulation point} of $A$ if for each $r\in\R^+$, there exists $\tilde z\in A$ such that $\tilde z - \tilde z_0$ is invertible and $\sharpnorm{\tilde z - \tilde z_0} < r$. 
\end{df}
\begin{thm}\label{thm_acc_point_of_zeroes}
Let $\tilde z_0\in\GenC$ and $r\in\R^+$. Let $A=\{\tilde z\in\GenC: \sharpnorm{\tilde z - \tilde z_0} < r\}$ and $u\in\tGen_H(A)$. Then the following are equivalent:
\begin{enumerate}
\item $\tilde z_0$ is a strict accumulation point of generalized zeroes of $u$
\item $D^k u(\tilde z_0)=0$, for each $k\in\N$
\item $u=0$ (in $\tGen(A)$).
\end{enumerate}
\end{thm}
\begin{proof}
$(1)\implies(2)$: by the sharp continuity of $u$, $u(\tilde z_0) = 0$. By theorem \ref{thm_holomorphic_charac}, $u(\tilde z) = (\tilde z - \tilde z_0)\sum_{n=1}^\infty a_n (\tilde z - \tilde z_0)^{n-1}$, for each $\tilde z\in A$ (with $a_n=D^n u(\tilde z_0)/n!$, $\forall n$). Let $u_1(\tilde z) := \sum_{n=1}^\infty a_n (\tilde z - \tilde z_0)^{n-1}$. By lemma \ref{lemma_convergence_radius}, this series converges for each $\tilde z\in A$, hence theorem \ref{thm_holomorphic_charac} implies that $u_1\in\tGen_H(A)$. Then $\tilde z_0$ is also a strict accumulation point of generalized zeroes of $u_1$. By the sharp continuity of $u_1$, $u_1(\tilde z_0) = a_1 = u'(\tilde z_0) = 0$. Inductively, one finds $D^k u (\tilde z_0) = 0$, for each $k\in\N$.\\
$(2)\Rightarrow (3)$: by theorem \ref{thm_holomorphic_charac}, $u(\tilde z) = \sum_{n=0}^\infty \frac{D^n u(\tilde z_0)}{n!} (\tilde z - \tilde z_0)^n = 0$, for each $\tilde z\in A$.\\
$(3)\implies(1)$: trivial.
\end{proof}
The conditions in the previous theorem, however, do not imply that $u=0$ on the `boundary of the convergence disc', as is shown by the example (cf.\ \cite{KS06}) $u_\eps(z)=z^{\lfloor\ln (\eps^{-1})\rfloor}$, for each $\eps\in(0,1)$, defining $u\in\Gen_H(\C)$ with $u(\tilde z)=0$ iff $\tilde z\approx 0$.

\begin{prop}[Analytic representatives]\label{prop_analytic_repres}
Let $A$ be a sharply bounded, internal subset of $\GenC$ with a sharply bounded representative $(A_\eps)_\eps$. Let $B\subseteq\GenC$ be an open set that contains an internal sharp neighbourhood of $A$, and $u\in\tGen_H(B)$.
Then $u$ has a representative $(u_\eps)_\eps$ with $u_\eps$ analytic on $A_\eps$, for each $\eps$.
\end{prop}
\begin{proof}
Let $m\in\N$ such that $u\in\tGen_H(\{\tilde z\in\GenC: (\exists \tilde\zeta\in A) (\abs[]{\tilde z-\tilde \zeta}\le \caninf^m)\})$ (such $m$ exists by lemma \ref{lemma_internal_nbd}). For each $\eps\in (0,1)$, $A_\eps + \frac{\eps^m}{2}\subset\C$ is compact. Hence we can cover $A_\eps + \frac{\eps^m}{2}$ by an open set $\omega_\eps$ consisting of a finite union of open squares of diameter $\eps^m/2$. Thus $\partial \omega_\eps$ is a polygon.
Given a representative $(u_\eps)_\eps$ of $u$ (with $u_\eps\in\Cnt[\infty](\C)$, for each $\eps$), we can define $v_\eps(z) := u_\eps(z)(1-\chi_\eps(z)) + \frac{\chi_\eps(z)}{2\pi i}\int_{\partial \omega_\eps} \frac{u_\eps(\zeta)}{\zeta - z}\,d\zeta$, where $\chi_\eps\in\Cnt[\infty](\C)$ with $\chi_\eps(z)=1$, for each $z\in A_\eps + \frac{\eps^m}{4}$ and $\chi_\eps(z)=0$, for each $z\in\C\setminus (A_\eps + \frac{\eps^m}{3})$, and $(\chi_\eps)_\eps\in\EMod(\C)$. Then $v_\eps$ is analytic on $A_\eps$. Further, by the Cauchy-Green theorem,
\[
\frac{\chi_\eps(z)}{2\pi i}\Big(\int_{\partial \omega_\eps} \frac{u_\eps(\zeta)}{\zeta - z}\,d\zeta - \int_{\partial B(z,\eps^{m+1})} \frac{u_\eps(\zeta)}{\zeta - z}\,d\zeta\Big) = \frac{\chi_\eps(z)}{\pi}\iint_{\omega_\eps\setminus B(z,\eps^{m+1})} \frac{\bar\partial u_\eps(\zeta)}{\zeta - z}\,dx\,dy
\]
defines a negligible net, since
\[
\sup_{z\in\C}\abs[\Big]{\partial^\alpha\Big(\chi_\eps(z)\iint_{\omega_\eps\setminus B(z,\eps^{m+1})} \frac{\bar\partial u_\eps(\zeta)}{\zeta - z}\,dx\,dy\Big)}\le \eps^{-M_\alpha}\meas(\omega_\eps)\sup_{z\in\omega_\eps}\abs{\bar\partial u_\eps(z)},
\]
which is negligible by proposition \ref{prop_NullA_internal} and the fact that $u\in\tGen_H([(\omega_\eps)_\eps])$. Also,
\begin{multline*}
\sup_{z\in \C}\abs{\partial^\alpha \Big(\frac{\chi_\eps(z)}{2\pi i}\int_{\partial B(z,\eps^{m+1})} \frac{u_\eps(\zeta)}{\zeta - z}\,d\zeta\Big)}\le\\ \eps^{-M_\alpha} \sup_{z\in A_\eps + \frac{\eps^{m}}{3},\,k\le\abs\alpha+1} \abs[\Big]{\int_{\partial B(z,\eps^{m+1})} \frac{u_\eps(\zeta)}{(\zeta - z)^k}\,d\zeta}
\le 2\pi \eps^{-M_\alpha -(m+1)\abs\alpha} \sup_{z\in A_\eps + \eps^m}\abs{u_\eps(z)},
\end{multline*}
hence $(v_\eps)_\eps\in \EMod(B)$. Again by the Cauchy-Green theorem,
\begin{multline*}
\sup_{z\in \C}\abs{u_\eps - v_\eps} =
\sup_{z\in \C}\abs[\Big]{\chi_\eps(z)\Big(u_\eps(z)-\frac{1}{2\pi i}\int_{\partial \omega_\eps} \frac{u_\eps(\zeta)}{\zeta - z}\,d\zeta\Big)}\\
\le \sup_{z\in A_\eps + \frac{\eps^{m}}{3}} \abs[\Big]{ \iint_{\omega_\eps} \frac{\bar\partial u_\eps(\zeta)}{\zeta - z}\,dx\,dy}
\le (2\pi + \meas(\omega_\eps))\sup_{z\in \omega_\eps}\abs{\bar\partial u_\eps(z)},
\end{multline*}
hence $u=[(v_\eps)_\eps]$ by proposition \ref{prop_GenA_pointvalues}(3).
\end{proof}


\begin{thebibliography}{10}
\bibitem{Aragona85}
J.\ Aragona,
{\em On the $\bar\partial$ Neumann problem for generalized functions}, J.\ Math.\ Anal.\ Appl.\ 110, 179--199, 1985.

\bibitem{AFJ05}
J.~Aragona, R.~Fernandez, S.~O.\ Juriaans,
{\em A discontinuous Colombeau differential calculus},
Monatsh.\ Math.\ 144, 13--29, 2005.

\bibitem{AFJO08}
J.~Aragona, R.~Fernandez, S.~O.\ Juriaans,
{\em Differential calculus and integration of generalized functions over membranes},
preprint {\tt arXiv:0809.4039}.

\bibitem{B90}
H.~Biagioni, {\em A Nonlinear Theory of Generalized Functions},
Lec.\ Notes Math.\ 1421, Springer, 1990.

\bibitem{Colombeau84}
J.-F.~Colombeau,
{\em New generalized functions and multiplication of distributions},
North-Holland, Amsterdam, 1984.

\bibitem{Colombeau85}
J.-F.~Colombeau,
{\em Elementary introduction to new generalized functions},
North-Holland, Amsterdam, 1985.

\bibitem{Garetto2005}
C.~Garetto,
{\it Topological structures in Colombeau algebras: topological $\widetilde\C$-modules and duality theory},
Acta Appl.~Math., 88(1), 81--123, 2005.

\bibitem{GKOS}
M.~Grosser, M.~Kunzinger, M.~Oberguggenberger, R.~Steinbauer,
{\it Geometric theory of generalized functions with applications to general relativity},
Kluwer 2001.

\bibitem{KS06}
A.~Khelif, D.~Scarpalezos,
{\em Zeros of generalized holomorphic functions},
Monatsh.\ Math.\ 149, 323--335, 2006.

\bibitem{OV_internal}
M.\ Oberguggenberger, H.\ Vernaeve,
{\em Internal sets and internal functions in Colombeau theory},
J.\ Math.\ Anal.\ Appl.\ 341, 649--659, 2008.

\bibitem{PSV06}
S.~Pilipovi\'c, D.~Scarpalezos, V.~Valmorin,
{\em Real analytic generalized functions},
Monatsh.\ Math., to appear.

\bibitem{S92}
D.~Scarpal\'{e}zos,
{\em Topologies dans les espaces de nouvelles fonctions g{\'e}n{\'e}ralis{\'e}es de Colombeau. $\GenC$-modules topologiques},
Universit\'e Paris 7, 1992.

\bibitem{HV_Banach}
H.\ Vernaeve,
Banach $\GenC$-algebras. Preprint, 2008.

\bibitem{HV_pointchar}
H.\ Vernaeve,
{\em Pointwise characterizations in generalized function algebras}, Monatsh.\ Math., to appear.
\end{thebibliography}
\end{document}